%% file: clt-scrambled-geom_2.tex
\documentclass[aos, preprint]{imsart}
\setattribute{journal}{name}{}
\RequirePackage[OT1]{fontenc}
\RequirePackage{amsthm,amsmath}
\RequirePackage[numbers]{natbib}
\RequirePackage[colorlinks,citecolor=blue,urlcolor=blue]{hyperref}
\usepackage{amsfonts,amssymb,graphicx,amsthm,nicefrac,mathtools}
\usepackage{enumerate, algorithm,algpseudocode, dsfont}
\usepackage{pstricks-add}
\usepackage{paralist}

\startlocaldefs
\newcommand{\norm}[1]{\left\lVert#1\right\rVert}

\include{newcommandsnotbook}

\newcommand{\bx}{\mathbb{X}}

\newcommand{\diam}{\mathrm{diam}}

\newcommand{\sumdot}{\text{\tiny$\bullet$}}

\newcommand{\bsk}{\boldsymbol{k}}

\newcommand{\bsn}{\boldsymbol{n}}
\newcommand{\bsw}{\boldsymbol{w}}
\newcommand{\bsv}{\boldsymbol{v}}

\DeclarePairedDelimiter{\floor}{\lfloor}{\rfloor}

\newcommand{\otos}{{1{:}s}}

\theoremstyle{definition}
\newtheorem{remark}{Remark}

\endlocaldefs

\begin{document}

\begin{frontmatter}

\title{Asymptotic Normality of Scrambled Geometric Net Quadrature}
\runtitle{Scrambled Geometric Net Quadrature}


\begin{aug}
\author{\fnms{Kinjal} \snm{Basu}\thanksref{t1}\ead[label=e1]{kinjal@stanford.edu}}
\and
\author{\fnms{Rajarshi} \snm{Mukherjee}\ead[label=e2]{rmukherj@stanford.edu}}

\thankstext{t1}{Supported by NSF Grant DMS-1407397}
\runauthor{K. Basu \and R. Mukherjee}

\affiliation{Stanford University}

\address{Department of Statistics\\
350 Serra Mall, Stanford CA 94305\\
\printead{e1}\\
\phantom{E-mail:\ }\printead*{e2}}

\end{aug}




\begin{abstract}
In a very recent work, Basu and Owen \cite{basu2015scrambled} propose the use of scrambled geometric nets in numerical integration when the domain is a product of $s$ arbitrary spaces of dimension $d$ having a certain partitioning constraint. It was shown that for a class of smooth functions, the integral estimate has variance $O( n^{-1 -2/d} (\log n)^{s-1})$ for scrambled geometric nets, compared to $O(n^{-1})$ for ordinary Monte Carlo. 
The main idea of this paper is to develop on the work by Loh \cite{loh:2003}, to show that the scrambled geometric net estimate has an asymptotic normal distribution for certain smooth functions defined on products of suitable subsets of $\real^d$.
\end{abstract}



\end{frontmatter}

\section{Introduction}
Quasi-Monte Carlo (QMC) sampling has been well developed for the purposes of integration over the unit cube $[0,1]^s$. Sampling over other regular shapes is a much more challenging problem that is receiving a lot more focus in the recent era. Measure preserving mapping from the unit cube to such shapes work very well for plain Monte Carlo. Unfortunately, the composition of the integrand with the mapping may fail to have even mild smoothness properties that QMC exploits \cite{basu:owen:2016}.

In this paper we consider the QMC integration via scrambled geometric nets as introduced in Basu and Owen \cite{basu2015scrambled}. The domains of interest are product spaces of the form $\cx^\otos := \prod_{j=1}^s\cx^{(j)}$ where each $\cx^{(j)}$ is a ``nice" bounded set in dimension $d$ (defined in Section \ref{sec:Fsmooth}). Integration over sets like triangles, spherical triangles, spheres, and disks are important in graphical rendering \citep{carlo2001state}. For instance, when $\cx^{(j)}$ is a triangle for  $j=1,2$, an integral of the
form $\int_{(T^2)^2}f(\bsx_1,\bsx_2)\rd \bsx_1\rd \bsx_2$ describes the potential
for light to leave one triangle and reach another. The function $f$
incorporates the shapes and relative positions of these triangles as
well as whatever lies between them. 

In \cite{basu2015scrambled}, the authors show that if each $\cx^{(j)}$ is a ``nice" bounded set in dimension $d$, then we can estimate
\[
\mu = \frac{1}{\vol( \cx^{\otos})} \int_{\cx^{\otos}} f(\bsx) d \bsx
\]
by the equal weight rule
\begin{align}
\label{mu_hat}
\hat{\mu} = \frac{1}{n} \sum_{i=1}^n f(\bsx_i),
\end{align}
where $\bsx_i$ are the points of a scrambled geometric net. In particular, \cite{basu2015scrambled} shows that $\hat{\mu}$ is unbiased for $\mu$ (see Proposition \ref{prop_unif_T}) and there exists a finite constant $C>0$ (depending on $f$) such that
\begin{align}
\label{up_bound}
\var(\hat{\mu}) \le   C\frac{(\log n)^{s-1}}{n^{1+2/d}},
\end{align}
under certain smoothness conditions on $f$ and a sphericity constraint on the partitioning of $\cx^\otos$.
This generalized the concept of scrambled nets which was introduced in a sequence of papers by Owen \cite{rtms, owensinum, smoovar, snxs, owen2008local}.

Although, this provides an upper bound on the mean squared error of $\hat{\mu}$, it is often of interest to obtain more precise results. In particular, one might seek to obtain asymptotically valid confidence interval type guarantees, as is the case with usual Monte Carlo integration. However, a simple variance upper bound guarantee as \eqref{up_bound} is not sufficient for such purposes. To obtain central limit theorem type results, it is crucial to obtain an asymptotically matching lower bound to \eqref{up_bound}, and thereafter, with this appropriate scaling, one can proceed to invoke standard tools to prove distributional convergence. This routine has been successfully realized in \cite{loh:2003}, where the author studied the asymptotic distribution of the scrambled net estimator over $[0,1]^s$. The proof relied on ensuring a suitable lower bound on the variance of the candidate $\hat{\mu}$ (matching up to constants to the upper bound obtained in \cite{smoovar}) and thereafter invoking the exchangeable pair argument of Stein's Method \citep{barbour2005introduction} to prove a central limit theorem for $\hat{\mu}$.

The main contribution of this paper is two fold. First, for a class of product spaces, of which $(T^2)^s$ is a special case, we show that the lower bound on $\var(\hat{\mu})$ matches the upper bound \eqref{up_bound} if $f$ satisfy certain smoothness assumptions. Secondly, we establish asymptotic normality of scrambled geometric net estimators on very general product spaces, for smooth functions where the lower bound to \eqref{up_bound} holds.  The main idea to prove the asymptotic normality of the scrambled geometric nets is based on the idea of Loh \cite{loh:2003}. 
Assuming that there exists a matching lower bound we show that we can generalize the proof of asymptotic normality from scrambled nets on $[0,1]^s$ to scrambled geometric nets on $\cx^\otos$. Loh states that the matching lower bound for scrambled nets can be obtained from the results in Owen \cite{smoovar}, however the same is not true for scrambled geometric nets. The lower bound does not follow from the main result in \cite{basu2015scrambled}. We show that, for a certain class of functions $\mathcal{F}_s$ (defined in Section \ref{sec:Fsmooth}), the lower bound on $\var(\hat\mu)$ matches the upper bound. That is, there exists a constant $c > 0$  (depending on $f$) such that for all $f \in \mathcal{F}_s$ and large enough $n$,
\begin{align}
\label{main_thm}
\var(\hat\mu) \ge c \frac{(\log n)^{s-1}}{n^{1 + 2/d}}.
\end{align}
For $s = 1$ and $\cx = T^2$, we present a separate proof of the lower bound on variance for a different class of functions, which might be of independent interest.

The rest of the paper is organized as follows.  Assuming familiarity with scrambled geometric nets, in Section \ref{sec:main_results}, we state the main results (asymptotic lower bound on variance and central limit theorem for $\hat{\mu}$) of the paper, followed by some numerical simulations in Section \ref{sec:num_results} to empirically validate the theoretical results.  For the sake of completeness, subsequently in Section \ref{sec:background} we discuss relevant backgrounds on geometric nets and scrambled geometric nets. Section \ref{sec:mrotos} contains discussions on ANOVA and multiresolution analysis on $\cx^\otos$ along with the form of $\var(\hat{\mu})$. Sections \ref{section:theorem1}, \ref{sec:proof_cor} and \ref{section:theorem2} are devoted to proving the main theorems and corollaries pertaining to the lower bound on $\var(\hat{\mu})$ and asymptotic normality of $\hat{\mu}$. Finally we collect proofs of several required lemmas in Appendices  \ref{app_lb_gain}, \ref{app:lemmas}, and \ref{app:lower_bound_triangles}.


We conclude this section with some related work on similar spaces. Previous work on using QMC for integration over simplices was due to Pillards and Cools, \cite{pillards2004theoretical, pill:cool:2005}. Instead of using transformations Basu and Owen \cite{Basu2014} developed two low-discrepancy point sets on the triangle. One is a lattice like construction which attained a discrepancy of $O(\log n/n)$. The other is a generalization of the van der Corput sequence by using the new theory of Koksma-Hwalka inequality on simplices \cite{bran:colz:giga:trav:2013}. Tractability results have been obtained in \cite{qmc:trac:simplices} for the $s$-fold product of the simplex $T^d =\{\bsx\in[0,1]^{d}\mid\sum_jx_j\le1\}$. For a survey of randomized QMC (RQMC) in general, see \cite{lecu:lemi:2002}. For an outline of QMC for computer graphics, see \cite{kell:2013}. 

\section{Main Results}
\label{sec:main_results}
In this section we state and discuss our main results pertaining to asymptotic normality of the scrambled geometric net estimator $\hat{\mu}$, as defined in \eqref{mu_hat}, with $\{\bsx_i : i = 1, \ldots, b^m\}$ being a scrambled $(0,m,s)$ geometric net in base $b$ (for a background on scrambled geometric nets refer to Section \ref{sec:background}). Assuming $\var(\hat{\mu}) := \sigma_{sgn}^2 > 0$, define,
\begin{align}
\label{Wdef}
W = \frac{\hat{\mu} - \mu}{\sigma_{sgn}}.
\end{align}
To avoid trivialities such as constant functions (which renders an identically zero variance) and rough functions (for which even an upper bound on the variance is unknown), we need to make certain assumptions on the class of functions $f$. Indeed, we shall show that if $f$ defined on $\cx^{1:s}$ belongs to a class of ``smooth functions" $\mathcal{F}_s$, then $\sigma_{sgn}$ satisfies a matching lower bound to \eqref{up_bound} and $W$ is asymptotically normal. 
\subsection{Smooth functions on $\cx^{\otos}$}
\label{sec:Fsmooth}
The results on scrambled nets and scrambled geometric nets are highly dependent on the smoothness properties of the function. Let $u \subseteq 1:s$, then denote $\partial^u f$ as the mixed partial derivative of $f$ taken once with respect to each $x_j$ for $j \in u$.

In \cite{owen2008local}, Owen defined a function $f$ on $[0,1]^s$ to be smooth if $\partial^u f(\bsx)$ is continuous on $[0,1]^s$ for all $u \subseteq 1:s$. Basu and Owen \cite{basu2015scrambled} generalize the same smoothness definition in \cite{owen2008local} for their results on scrambled geometric nets. The smoothness conditions depend on the underlying domain $\cx^\otos$. If the space is Sobol' extensible then the function $f$ is said to be smooth if $\partial^{1:ds}f$ is continuous on $\cx^\otos$. If the domain is not Sobol' extensible the authors assume that $f \in C^{ds}(\cx^{\otos})$. Although, these mild conditions are enough to show the upper bound of $\sigma_{sgn}^2$, they might not be enough to show a matching lower bound as per Theorem \ref{thm:lower-bound}. For details on Sobol' extensible sets, see \cite{basu2015scrambled}. Also, for the proof of upper bound of $\sigma_{sgn}^2$, Basu and Owen \cite{basu2015scrambled} assume a sphericity constraint for the construction of the scrambled geometric net. See Section \ref{sec:background} for more details.

In view of the above discussion, throughout the rest of the paper we assume $\cx^\otos$ to be a Sobol' extensible region, with a recursive partition satisfying the sphericity constraint.  We define a sub-class of smooth functions $\mathcal{F}_s$ on $\cx^\otos$ as follows. For any $u \subseteq \otos$, we denote the order $u$ gradient of $f$ as a $d^{|u|}$ dimensional vector, $\nabla^u f(\bsx)$. Formally, the coordinates of $\nabla^u f(\bsx)$ are $\frac{\partial^{|u|}f (\bsx) }{\prod_{j \in u} \partial x_{ji_j}}$ where $i_j \in \{1, \ldots, d\}$ for each $j$, stacked in some prefixed lexicographic order. A smooth function is a function $f$ such that mixed partial gradient satisfies the following Lipschitz condition. 

\begin{definition}
\label{smoodef}
Let $\cx^\otos$ be Sobol' extensible. A real-valued function $f$ on $\cx^{\otos}$ is smooth if for all $u \subseteq \otos$,
$$\norm{\nabla^u f(\bsx) - \nabla^u f(\bsx^*)} \le B \norm{ \bsx - \bsx^*}^\beta$$ for some finite $B \ge 0$ and $\beta \in (0,1]$ for all $\bsx, \bsx^* \in \cx^\otos$.
\end{definition}

\begin{definition}\label{smoothf_def}
Let $\cx^\otos$ be Sobol' extensible. Define $\mathcal{F}_s$ as the class of all smooth functions $f$ on $\cx^\otos$ such that for all $u \subseteq \otos$,
\begin{align*}
 \norm{ \int_{\cx^\otos} \nabla^u f (\bsx) \rd \bsx }^2  > 0,
\end{align*}
where the above integral is done coordinate-wise.
\end{definition}

\begin{remark}
	Under additional smoothness assumptions, our results hold even when the domain is not Sobol' extensible. However, for sake of compactness of proof we only work with Sobol' extensible sets. 
\end{remark}

 \subsection{Lower Bound on Variance}
 As mentioned earlier, a crucial step towards understanding the asymptotic distribution of $W$ is obtaining a matching lower bound to \eqref{up_bound}. To state the result we need the following notation. Let $\bsn_{c_j}$ and $\bsw_{j}$ denotes the center of $\cx_{j,(k_j,t_j,c_j)}$ and $\cx_{j,(k_j,t_j)}$ which satisfy,
\begin{align}
\label{center_property}
\int_{\cx_{j,(k_j,t_j)}} \langle \bsx - \bsw_{j} , \boldsymbol{\delta} \rangle \rd \bsx = 0, \qquad \int_{\cx_{j,(k_j,t_j,c_j)}} \langle \bsx - \bsn_{c_j} , \boldsymbol{\delta} \rangle \rd \bsx = 0.
\end{align}
for any fixed vector $\boldsymbol{\delta}$. Here $\cx_{j,(k_j,t_j,c_j)}$ and $\cx_{j,(k_j,t_j)}$ denotes the narrow and wide cells of $\cx^{(j)}$ as defined in Definition \ref{cell_def_2} and \eqref{cellDef}. For notational simplicity we hide the dependence on $k_j$ and $t_j$ for $\bsn_{c_j}, \bsw_j$. Finally let,
\begin{align}
\label{eq:Aj}
A_j = \sum_{c_j = 0}^{b-1} (\bsn_{c_j} - \bsw_j)(\bsn_{c_j} - \bsw_j)^{T}, 
\end{align}
and $\lambda_1(A_j)$ be the minimum eigenvalue of $A_j$.
With this, we are ready to state our first main result of the paper regarding the matching lower bound on variance. 

\begin{theorem}
\label{thm:lower-bound}
If $f \in \mathcal{F}_s$ and $\lambda_1(A_j) \ge \tilde{c}b^{-2k_j/d}$ for all $j = 1, \ldots, s$, $\tilde{c} > 0$, then there exists a positive constant $c$ such that
\begin{align}
\label{eq:sigma_sgn}
\sigma_{sgn}^2 \ge c \frac{(\log n)^{s-1}}{n^{1+2/d}}
\end{align}
for all sufficiently large $m$.
\end{theorem} 
Indeed to prove appropriate lower bounds on variance by applying Theorem \ref{thm:lower-bound} above, one needs to show that $\lambda_1(A_j) \ge \tilde{c}b^{-2k_j/d}$. This is often a rather space specific property. In the following two corollaries we show this indeed holds for the cases of interval or triangle.
\begin{corollary}
\label{cor:triangle}
Let $b = 4$ and  $\cx = T^2$ and $f \in \mathcal{F}_s$ for $s \le b$. Then, there exists positive constants $c, C$ such that 
\begin{align*}
c \frac{m^{s-1}}{b^{2m}} \le \var(\hat{\mu}) \le C \frac{m^{s-1}}{b^{2m}}
\end{align*}
as $m \rightarrow \infty$.  
\end{corollary}

\begin{remark}
To prove the above result for general $s$, one can choose $b = 4^\ell$ such that $ b \ge s$ and follow the similar proof technique as in Section \ref{sec:proof_cor}.
\end{remark}

\begin{corollary}\label{cor:square}
Let $b \ge \max(s,2), \cx = [0,1]$ and $f \in \mathcal{F}_s$. Then, there exists positive constants $c, C$ such that 
\begin{align*}
c \frac{m^{s-1}}{b^{3m}} \le \var(\hat{\mu}) \le C \frac{m^{s-1}}{b^{3m}}
\end{align*}
as $m \rightarrow \infty$.  
\end{corollary}

Loh \cite{loh:2003} stated that for the scrambled net, the lower bound result follows from Theorem 2 in Owen \cite{smoovar}. However, in a later paper Owen \cite{owen2008local}, acknowledges that there was an error in Lemma 1 of \cite{smoovar} and appropriately corrects the proof with milder smoothness assumptions. However, unfortunately the lower bound result does not automatically follow from the new proof in \cite{owen2008local}. Corollary \ref{cor:square} corrects the proof for $\cx = [0,1]$.

The examples in Corollaries \ref{cor:triangle} and \ref{cor:square}, naturally give rise to a class of domains which are invertible linear transformations of the square and the triangle, and for which similar variance upper and lower bounds continue to hold.  In particular, a general result in this direction can be described as follows. For the sake of simplicity, consider $s=1$ and a domain $\cx\subset \real^d$ for which the condition of Theorem \ref{thm:lower-bound} hold. Then consider a $d\times d$ invertible matrix $B$ such that the eigenvalues of $BB^T$ are bounded away from $0$. Then it is easy to show that, the resulting domain $B(\cx)$ and the corresponding recursive split induced by the image of the recursive split of $\cx$ under $B$, satisfy the condition of Theorem \ref{thm:lower-bound}. As a result, the lower bound on $\sigma^2_{sgn}$ continue to hold on the transformed domain $B(\cx)$, by a simple application of Theorem \ref{thm:lower-bound}. In particular, one naturally obtains results analogous to Corollaries \ref{cor:triangle} and \ref{cor:square} for non-right-angle triangles and parallelograms respectively.

\begin{remark}
Note that the condition on $\lambda_1$ in Theorem \ref{thm:lower-bound} is not necessary. To see an example, consider $s = 1, b = 2$ and $\cx = [0,1]^2$ which we split alternately by horizontal and vertical splits. In the notation of Theorem \ref{thm:lower-bound}, it is easy to show $\lambda_1(A_1) = 0$ (owing to the collinearity of $\bsw_{k_1, t_1}, \bsn_{k_1, t_1, 0}$ and $\bsn_{k_1,t_1,1}$ for all $t_1 = 0, \ldots, b^{k_1} - 1$ and all $k_1$). Now, following the arguments of Section \ref{sec:specific_domain} and using the second eigenvalue and the corresponding eigenvector explicitly, one can easily show that we get the same lower bound as in \eqref{eq:sigma_sgn} provided $\norm{\nabla^{\{1\}} f}^2 > c$ for some $c > 0$.
\end{remark}


\subsection{Asymptotic Normality}
Equipped with the lower bound in Theorem \ref{thm:lower-bound}, we are now ready to state the  promised asymptotic normality of $W$ in the next theorem.
\begin{theorem}
\label{thm:convergence}
Let $b \ge \max(s, d, 2)$, $f \in \mathcal{F}_s$ and $W$ be as defined in \eqref{Wdef}. If \eqref{eq:sigma_sgn} holds, then $W \rightarrow \mathcal{N}(0,1)$ in distribution as $m \rightarrow \infty$.
\end{theorem}

Loh \cite{loh:2003} proved a version of Theorem \ref{thm:convergence} for the scrambled net. The main idea of the proof was to create a $\tilde{W}$ satisfying $W - \tilde{W} = o_p(1)$ and there after creating an exchangeable pair $(\tilde{W}, \tilde{W}^*)$ to show that $\tilde{W} \rightarrow \mathcal{N}(0,1)$ in distribution as $m \rightarrow \infty$.  We adopt the same proof technique. Indeed, if  we assume that Theorem \ref{thm:lower-bound} holds, it is not hard to see that our proof will follow along the exact same lines of Loh \cite{loh:2003} since most of the results from \cite{loh:2003} do not depend on the domain being $[0,1]^s$. They mostly depend on the properties of the scrambled net. Since the scrambled geometric net also has the same properties it is not hard to see that almost all results from Loh \cite{loh:2003} will go through except for Lemma 3 in that paper. 
The result corresponding  to Lemma 3 in \cite{loh:2003} is our Lemma \ref{lem:cor_lemma_3} in Appendix \ref{app_thm}.

\begin{remark}
We note that as a consequence of Corollary \ref{cor:square} and Theorem \ref{thm:convergence}, Theorem 2 in \cite{loh:2003} follows as a special case.
\end{remark}

\section{Numerical Results}
\label{sec:num_results}
Before we describe the theoretical background, we present some numerical studies which verify the convergence results. 
Throughout this section and in practice to construct an estimate of the variance of the scrambled geometric net estimator we use $N$ independent replications of the sampling scheme. For each $\ell = 1, \ldots, N$, let $\hat{\mu}_\ell$ denote the scrambled geometric net estimator of the integral based on the $\ell^{\mathrm{th}}$ sample of points from the domain of interest. Then, we estimate the variance of the estimator using
\begin{align}
\label{eq:estimator_var}
\hat{\sigma}_{sgn}^2 = \frac{1}{N-1}\sum_{\ell = 1}^N\left(\hat{\mu}_\ell - \bar{\hat{\mu}}\right)^2,
\end{align}
where $\bar{\hat{\mu}} = \frac{1}{N}\sum_{\ell = 1}^N \hat{\mu}_\ell $.

Using this estimate of the variance and Theorem \ref{thm:convergence} we can construct the $(1-\alpha)100\%$ confidence interval for $\mu$ as 
\begin{align}
\label{eq:ci}
\left(\hat{\mu} - z_{\frac{\alpha}{2}}\hat{\sigma}_{sgn}, \hat{\mu} + z_{\frac{\alpha}{2}}\hat{\sigma}_{sgn}\right),
\end{align}
where $z_{\frac{\alpha}{2}}$ is the $(1 - \frac{\alpha}{2})$-th quantile of the normal distribution.

\begin{example}
In this example we work with the integrand given by $f(\bsx_1, \bsx_2) = x_{11}x_{12}^2 - x_{21}^3x_{22}^4$ defined over the product space $T^2 \times T^2$, where $\bsx_i = (x_{i1},x_{i2}) \in T^2$ for $i = 1,2$. Figure \ref{fig:sigma_eg1} shows the natural logarithm of the estimated variance as the sample size increases by taking $N = 300$ independent replications. The solid line shows the estimated variance of the scrambled geometric net estimator and the dashed line is used to display the estimated variance of the Monte Carlo estimator. It can be observed from Figure \ref{fig:sigma_eg1} that the estimated variance of scrambled geometric net estimator decays at the rate $\log (n)/n^2$, whereas the estimated variance of the Monte Carlo estimator behaves like $C/n$ for some $C > 0$.  

Note that the exact value of this integral is $\mu = 41/5040$. To see asymptotic normality, we plot the smoothed histogram of $W_\ell$, for $\ell = 1,\ldots, N$, where
\begin{align*}
W_\ell = \frac{\hat{\mu}_\ell - \mu}{\hat{\sigma}_{sgn}}.
\end{align*}
This is shown in Figure \ref{fig:normality}. It can be seen from Figure \ref{fig:normality} that as $m$ increases we get closer and closer to normality.

Using new independent replications, we construct 100 different confidence intervals for $\mu$ at level $\alpha = 0.05$. This is shown in Figure \ref{fig:ci}. Throughout this simulation we keep $n = 4^6$. The intervals shown in red fail to contain the true value of $\mu$, and as a result it can be seen that we have desired control over the coverage of the confidence intervals.
\end{example}

\begin{figure}[!h]
\centering
\includegraphics[width = \linewidth]{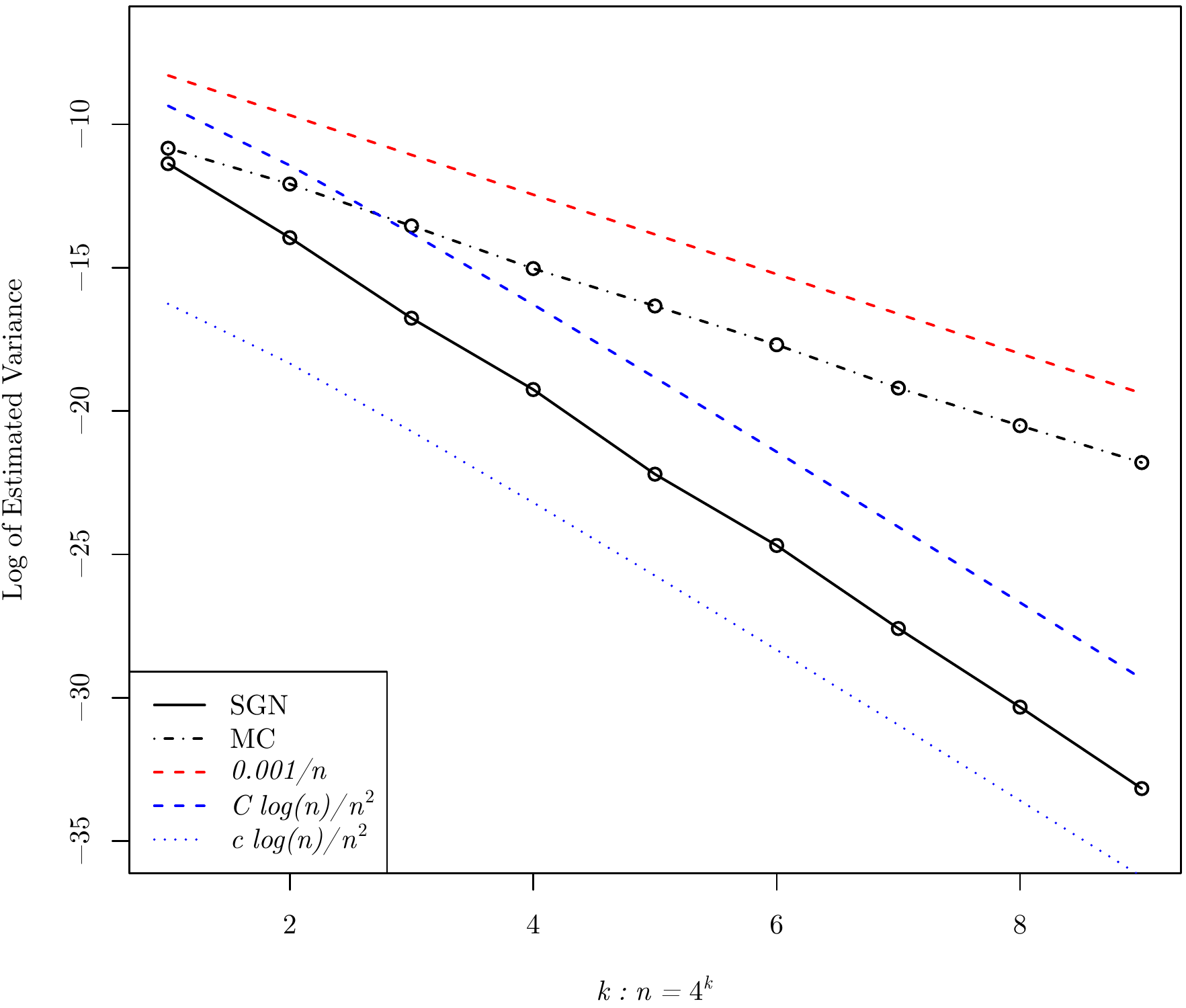}
\caption{\label{fig:sigma_eg1}
Decay of estimated variance as a function of sample size in a log-log scale. The solid and dashed black line show the log of estimated variance using scrambled geometric nets and Monte Carlo sampling respectively.}
\end{figure}

\begin{figure}[!h]
\centering
\includegraphics[width = \linewidth]{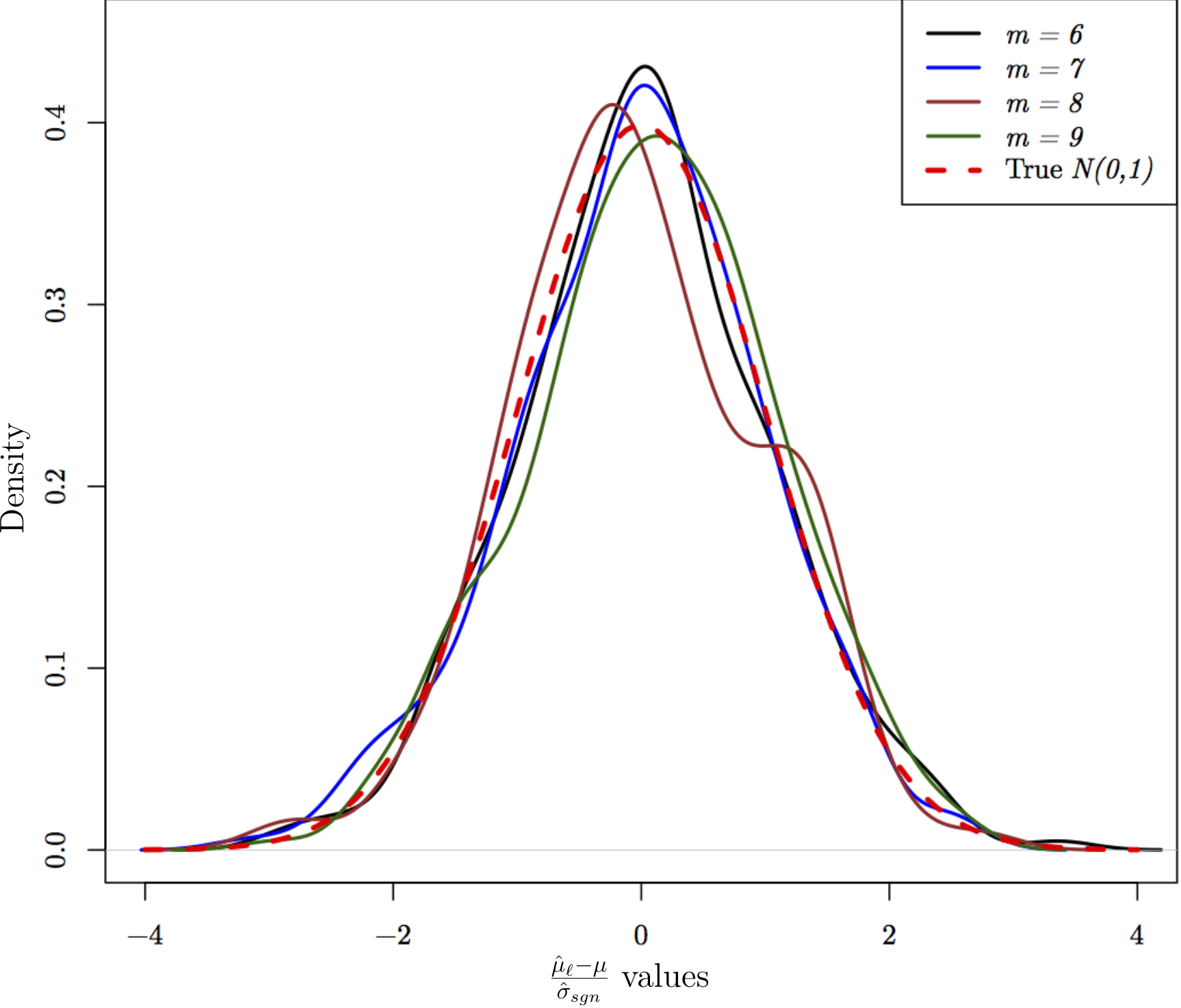}
\caption{\label{fig:normality}
Empirical verification of asymptotic normality for scrambled geometric net estimator. The $x$-axis shows the centered (with the true mean $\mu$) and scaled (with the estimated standard deviation) scrambled geometric net estimator.}
\end{figure}

\begin{figure}[!h]
\centering
\includegraphics[width = \linewidth]{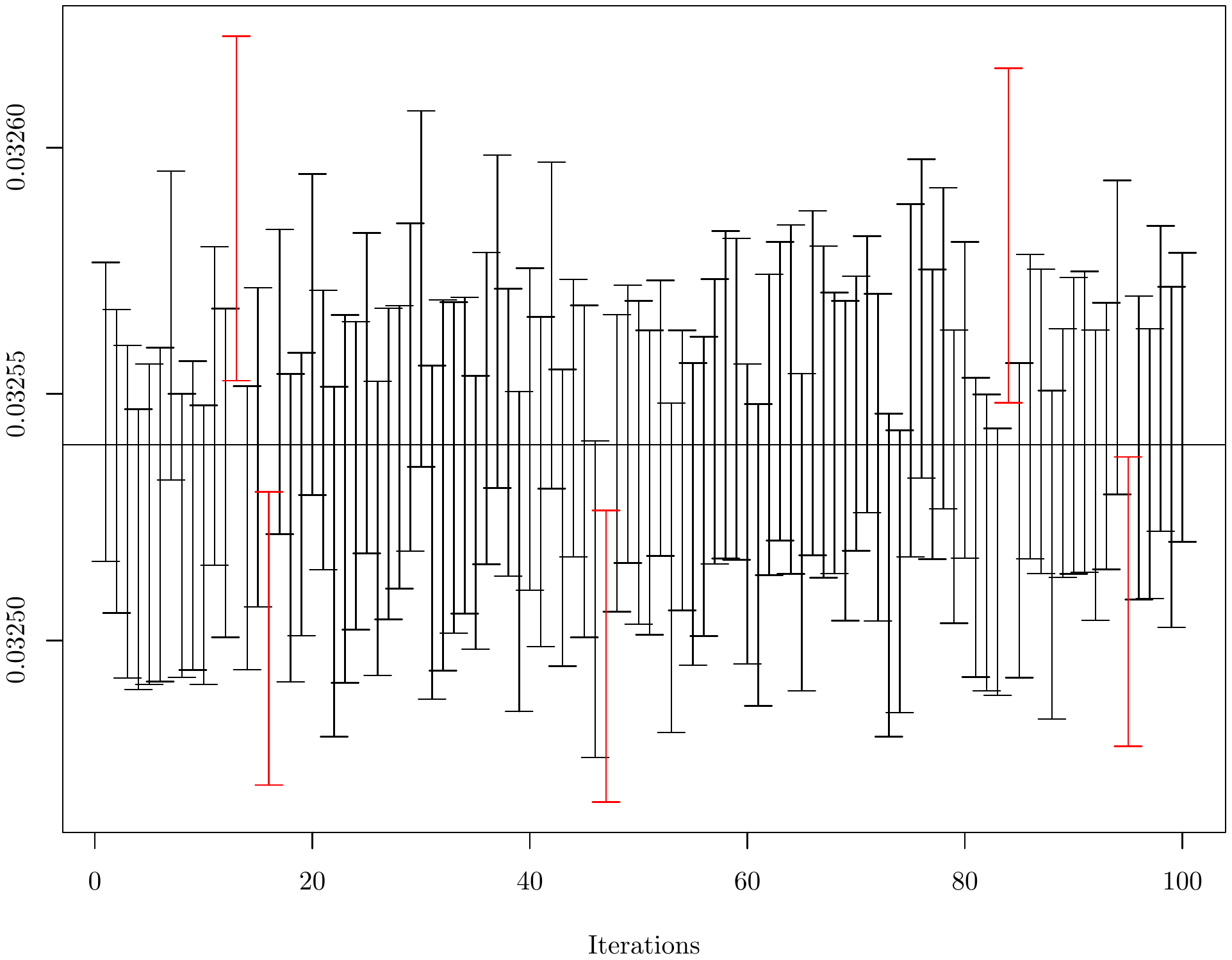}
\caption{\label{fig:ci}
100 replications of 95\% confidence intervals for $\mu$ constructed using scrambled geometric net estimators. The true value is denoted by the horizontal black line. The confidence intervals which do not contain the true $\mu$ are shown in red.}
\end{figure}

\begin{example}
In this example we use the integrand $$f(\bsx_1, \bsx_2) = x_{11}x_{12}x_{21}x_{22} \exp{x_{11}x_{12}x_{21}x_{22} },$$
defined on $T^2 \times T^2$ which is difficult to integrate analytically. A similar function was considered in \cite{owen2008local}. For this function, we compare the confidence interval formed by the scrambled geometric net estimator and the Monte Carlo estimator in Figure \ref{fig:mc_vs_sgn}. Here too we keep $n = 4^6$. The solid and dashed lines show the confidence interval using scrambled geometric nets and Monte Carlo respectively. It can be easily seen that, as predicted by the theoretical results, we get systematically much smaller confidence interval using scrambled geometric nets than Monte Carlo. 
\end{example}

\begin{figure}[!h]
\centering
\includegraphics[scale = 0.61]{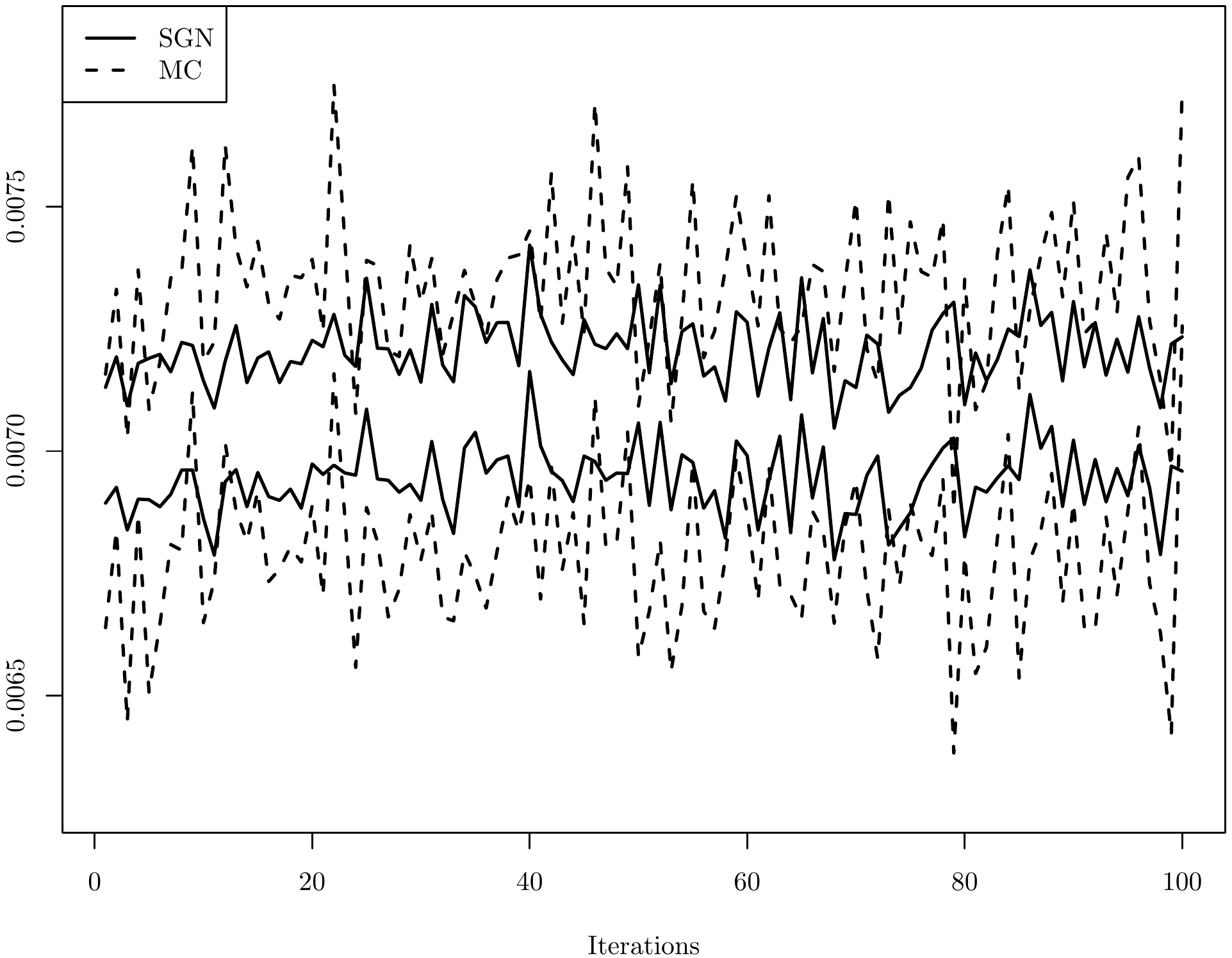}
\caption{\label{fig:mc_vs_sgn}
Confidence Intervals for $\mu$ generated by the two different sampling techniques. The solid and dashed lines show the confidence intervals using scrambled geometric nets and Monte Carlo respectively with $n = 4^6$. }
\end{figure}

\section{Background on Scrambled Geometric Nets}
\label{sec:background}
Before proving the main results, we discuss some necessary background on digital nets and the scrambling algorithm. We proceed through a sequence of definitions. Let $b\ge2$ be an integer base, $s\ge1$ is an integer dimension and
$\ints_b=\{0,1,\dots,b-1\}$.

\begin{definition}
For $k_j\in\natu_0$ and $c_j\in\ints_{b^{k_j}}$ for $j=1,\dots,s$, the set 
$$
\prod_{j=1}^s\Bigl[ \frac{c_j}{b^{k_j}},\frac{c_j+1}{b^{k_j}}\Bigr) 
$$
is a $b$-adic box of dimension $s$. 
\end{definition}

\begin{definition}
For integers $m\ge t\ge0$,
the points $\bsx_1,\dots,\bsx_{b^m}\in[0,1]^s$ 
are a $(t,m,s)$-net in base $b$
if every $b$-adic box of dimension $s$ with volume $b^{t-m}$
contains precisely $b^t$ of the $\bsx_i$. 
\end{definition}

Owen \cite{rtms} introduced the idea of nested uniform scramble of $(t,m,s)$-nets as follows.
Let $a\in[0,1]$
have base $b$ expansion
$a=\sum_{k=1}^\infty a_{k}b^{-k}$  where $a_{k}\in\ints_b$. If $a$
has two base $b$ expansions, without loss of generality we take the one with a tail of $0$s,
not a tail of $b-1$s.
We apply random permutations to the digits $a_k$ yielding $u_k\in\ints_b$
and deliver $u=\sum_{k=1}^\infty u_{k}b^{-k}$.
There are many different ways to choose the permutations \citep{altscram}.

Consider a sequence of uniform random permutations of $\mathbb{Z}_b$
\begin{align*}
\{ \pi_{j\sumdot }, \pi_{j\sumdot a_{j1}}, \pi_{j \sumdot a_{j1}, a_{j2}}, \ldots, : a_{jk} \in \mathbb{Z}_b, 1 \le j \le s, k = 1,2,\ldots   \}, 
\end{align*}
where all of permutations are independent. In a nested uniform scramble of $\bsa=(a_1,\dots,a_s)\in[0,1]^s$ we apply the above set of permutations to all components of $\bsa$ to get $u_{j,k+1}=\pi_{j\sumdot a_{j1}a_{j2},\dots,a_{jk}}(a_{j,k+1})$ for $j = 1,\ldots, s, \;k \in \natu$. We return $\bsu = (u_1, \ldots, u_s)$ where $u_j = \sum_{k=1}^\infty u_{j,k}b^{-k}$. The nested uniform scramble of a set of $n$ points $\bsa_1,\dots,\bsa_n\in[0,1]^s$ applies the same set of
permutations to all $n$ of those points.

There are several important results related to digital nets and scrambling. For details we refer to the series of papers by Owen \cite{rtms, owensinum, smoovar, snxs, owen2008local} and \citep{dick:pill:2010, nied87}.

\subsection{Geometric Transformation and Scrambled Geometric Nets}
Scrambled geometric nets are created by a transformation of a scrambled $(t,m,s)$-net on $[0,1]^s$ to the domain of interest. To explain the transformation we introduce the concept of recursive splits as in \cite{basu2015scrambled}.
For sake of completeness, we restate the definitions.

\begin{definition}
Let $\cx\subset\real^d$ have finite and positive volume. 
A $b$-fold split of $\cx$ is a collection of Borel sets $\cx_a$ for $a\in\ints_b$
with $\cx =\cup_{a=0}^{b-1}\cx_a$, $\vol(\cx_a)=\vol(\cx)/b$ for 
$a\in\ints_b$, and $\vol(\cx_a\cap\cx_{a'})=0$ for 
$0\le a<a'<b$. 
\end{definition}
\begin{definition}
Let $\cx\subset\real^d$ have finite and positive volume. 
A recursive $b$-fold split of $\cx$ is a collection $\bx$ of
sets consisting of $\cx$ and exactly one
$b$-fold split of every set in the collection. The members of $\bx$
are called cells.
\end{definition}

See \cite{basu2015scrambled} for explicit splits of triangles, spherical triangles and discs. Given a set $\cx$ and a recursive splitting of it in base $b$, to define the transformation we begin by considering any point $u \in [0,1]$. We can expand $u$ in base $b$ as $0.u_1u_2\ldots$.  Corresponding to this $u$ we define a sequence of sets
$$\cx_{1:K} = \cx_{u_1,u_2, \ldots, u_K}.$$
Then $\bsx := \phi(u)$ is any point in
$\cap_{K=1}^\infty \cx_{1:K}$.
The volume of $\cx_{1:K}$ is $b^{-K}$ which converges
to $0$ as $K\to\infty$.  To get a unique limit $\bsx$,
Basu and Owen \cite{basu2015scrambled} use the notion of a sequence of sets converging nicely
to a point \cite{stro:1994}. A recursive split in base $b$ is said to be convergent if for every infinite sequence $u_1, u_2, \ldots, \in \mathbb{Z}_b$, the cells $\cx_{u_1,\dots,u_K}$ converge nicely to a point as $K \rightarrow \infty$. A simple sufficient condition for a convergent split is the sphericity constraint.

\begin{definition}
Let $\bx$ be a recursive split of $\cx\in\real^d$ in base $b$.
Then $\bx$ satisfies the {\sl sphericity condition} if 
there exists a positive constant $C<\infty$ such that
$\diam(\cx_{u_1,\dots,u_k})\le Cb^{-k/d}$ holds for all cells
$\cx_{u_1,\dots,u_k}$ in $\bx$.
\end{definition}

\begin{definition}
Given a set $\cx\subset\real^d$ and a convergent recursive split $\bx$
of $\cx$ in base $b$, the $\bx$-transformation of $[0,1]$ is the function
$\phi = \phi_\bx:[0,1]\to\cx$ given by
$\phi(u) = \lim_{K\to\infty} \cx_{u_1,u_2,\dots,u_K}$
where $u$ has the base $b$ representation $0.u_1u_2\dots$.
\end{definition}

Now we are at a stage to define digital geometric nets in $\cx^s$ via recursive splittings. For ease of readability, we follows the same notation as in \cite{basu2015scrambled}.

For $s\in\natu$, we represent the set $\{1,2,\dots,s\}$ by $1{:}s$. 
For $j\in1{:}s$ we have bounded sets $\cx^{(j)}\subset\real^{d_j}$ with 
$\vol(\cx^{(j)})=1$. 
For sets of indices $u\subseteq1{:}s$, the complement $1{:}s \backslash u$ is denoted by $-u$. 
We use $|u|$ for the cardinality of $u$.
The Cartesian product of $\cx^{(j)}$ for $j\in u$ is denoted $\cx^u$. 
A vector $\bsx\in \cx^{1{:}s}$ has components $\bsx_j\in\cx^{(j)}$. 
The vector in $\cx^u$ with components $\bsx_j$ for $j\in u$ is 
denoted $\bsx_u$.  A point in $\cx^\otos$ has $\sum_{j=1}^sd_j$ components. 
We write it as 
$\bsx=(\bsx_1,\bsx_2,\dots,\bsx_s)$, where each $\bsx_j$ has $d_j$ components for $j = 1, \ldots, s$.

\begin{definition}
\label{cell_def_2}
For $j=1,\dots,s$, let $\bx_j$ be a recursive split of $\cx^{(j)}$ in a common base $b$.
Denote the cells of $\bx_j$ by $\cx_{j,(k,t)}$ for $k\in\natu$ and $t\in\ints_{b^k}$.
Then a $b$-adic cell for these splits is a Cartesian product of the form
$\prod_{j=1}^s\cx_{j,(k_j,t_j)}$
for integers $k_j\ge0$ and $t_j\in\ints_{b^{k_j}}$.
\end{definition}

\begin{definition}
Let $\cx^{(j)}\subset\real^{d_j}$ have volume $1$ for $j\in\otos$
and let $\bx_j$ be a recursive split of $\cx^{(j)}$ in a common base $b$.
For integers $m\ge t\ge0$,
the points $\bsx_1,\dots,\bsx_{b^m}\in\cx^\otos$ 
are a geometric $(t,m,s)$-net in base $b$
if every $b$-adic cell of volume $b^{t-m}$
contains precisely $b^t$ of the $\bsx_i$. 
\end{definition}

Basu and Owen \cite{basu2015scrambled} prove the following results regarding scrambled geometric nets. 

\begin{propo}
Let $\bsa_1,\dots,\bsa_n$ be a $(t,m,s)$-net in base $b$.
Let $\bsu_1,\dots,\bsu_n$ be a nested uniform scramble of
$\bsa_1,\dots,\bsa_n$. For $j\in\otos$, let $\bx_j$ be a recursive
base $b$ split of the unit volume set 
$\cx^{(j)}\subset\real^{d_j}$ with $\bx_j$-transformation $\phi_j$.
Then $\bsx_i = \phi(\bsu_i)$ (componentwise) is a scrambled geometric
$(t,m,s)$-net in base $b$ with probability one.
\end{propo}

\begin{propo}
\label{prop_unif_T}
Let $\cx^{(j)}\subset\real^{d_j}$ with $\vol(\cx^{(j)})=1$ for $j\in\otos$
have convergent recursive splits $\bx_j$ in bases $b_j\ge 2$ with corresponding
transformations $\phi_j$.
Let $\bsa\in[0,1]^s$ and let $x_j$ be a base $b_j$ nested uniform scramble of $a_j$.
Then $\phi(\bsx)=(\phi_1(x_1),\dots,\phi_s(x_s))\sim\dustd(\cx^\otos)$.
\end{propo}

\subsection{Illustration on $(T^2)^s$}
It is easiest to visualize the scrambled geometric net when the domain is a product of triangles. Denote $T_j^2$ as the $j$-th triangle, where $T^2 = \{ \bsx \in \real^2 : x_1, x_2 \ge 0, x_1 + x_2 \le 1\}$ and we aim to construct a $n$ point scrambled geometric net on $(T^2)^s = \prod_{j = 1}^sT_j^2$.  

Let  $\bsa_1,\dots,\bsa_n$ denote a $(t,m,s)$-net in base $4$. Corresponding to $\bsa_1,\dots,\bsa_n$, let $\bsu_1,\dots,\bsu_n$ denote the nested uniform scramble where $\bsu_i = (u_{i1},\ldots, u_{is})$ for $i=1,\dots,n$. We expand $u_{ij}$ in base $4$ to get $u_{ij}=\sum_{k=1}^{\infty}u_{ijk}4^{-k}$, where $u_{ijk} \in \mathbb{Z}_4$. We now map the initial sequence $(u_{ij1},u_{ij2},\dots,u_{ijK})$ to one of the four sub-triangles of $T_j^2$ of volume $4^{-K}$, illustrated in the left panel of Figure~\ref{fig:subdiv}. We use a slightly different labelling than in \cite{basu2015scrambled} to help us in the later proofs.

\begin{figure}[!h]
\centering
\includegraphics[scale = 0.6]{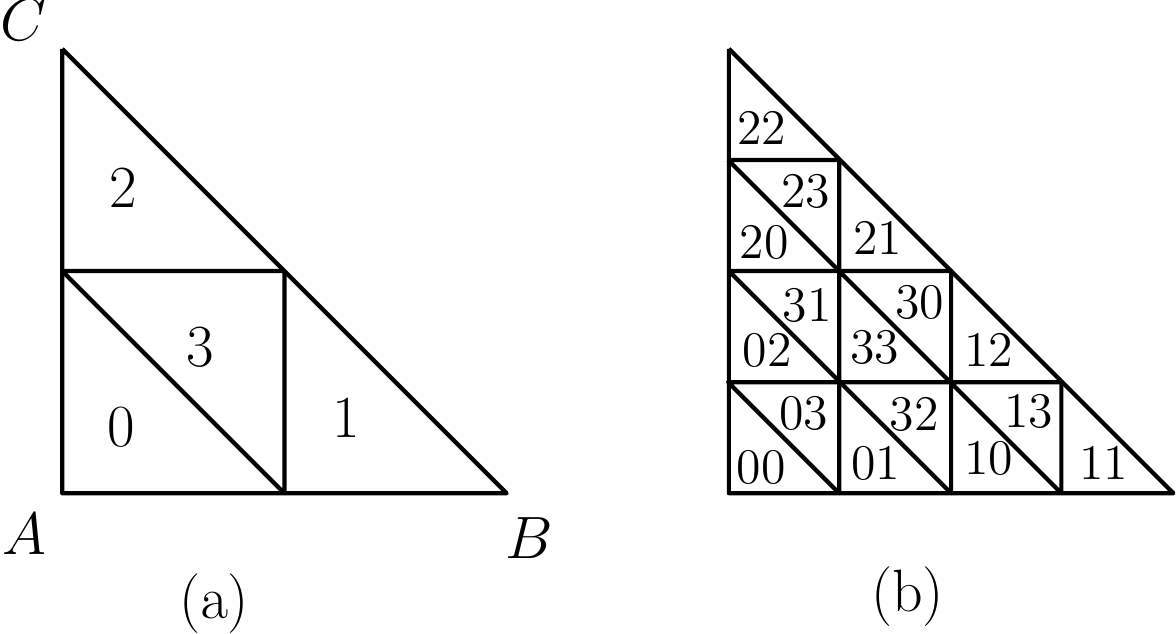}
\caption{\label{fig:subdiv}
A labeled subdivision of $\Delta(A,B,C)$ into $4$ and then
$16$ congruent subtriangles.}
\end{figure}


The subtriangle is $T_j^2( u_{ij1})$ where if $T_j^2$ has corners $A$, $B$ and $C$
then
$$
T_j^2(u) = 
\begin{cases}
\Delta\bigl( A, \frac{A+B}2,\frac{A+C}2\bigr), & u=0\\[0.5ex]
\Delta\bigl( \frac{B+A}2, B,\frac{B+C}2\bigr), & u=1\\[0.5ex]
\Delta\bigl( \frac{C+A}2, \frac{C+B}2,C\bigr), & u=2\\[0.5ex]
\Delta\bigl( \frac{B+C}2,\frac{A+C}2,\frac{A+B}2\bigr), & u=3.
\end{cases}
$$

Next, $T_j^2(u_{ij1},u_{ij2})=(T_j^2(u_{ij1}))(u_{ij2})$ corresponding
to digit $u_{ij2}$ within $T_j^2(u_{ij1})$ as shown in the right panel
of Figure~\ref{fig:subdiv}.
In general, $T_j^2(u_{ij1},\dots,u_{ijk}) = (T_j^2(u_{ij1},\dots,u_{ij(k-1)}))(u_{ijk})$.
This process maps the sequence $(u_{ij1},\dots,u_{ijK})$ to the triangle
$T_j^2(u_{ij1},u_{ij2},\dots,u_{ijK})$. The point $x_{ij}=\phi_j(u_{ij})$ is the center of triangle $T_j^2(u_{ij1},u_{ij2},\dots,u_{ijK})$ and thus $\bsx_i = (x_{i1},\dots,x_{is})\in (T^2)^s$. These points $\bsx_1, \ldots, \bsx_n$ form the scrambled geometric net on $(T^2)^s$. For more details on the triangular construction we refer to \cite{basu2015scrambled, Basu2014}.

\section{ANOVA and multiresolution for $\cx^\otos$}\label{sec:mrotos}

There is a well known analysis of variance (ANOVA) for $[0,1]^s$. In \cite{basu2015scrambled}, the authors present a general theory for $\cx^\otos$. For sake of completeness and to introduce the related notation which we need in our proofs we describe the ANOVA for $\cx^\otos$ and a multiresolution analysis of $L^2(\cx^\otos)$.

\subsection{ANOVA of $\cx^\otos$}
Let $f\in L^2(\cx^{1{:}s})$ and $u \subseteq \otos$. The ANOVA decomposition of $f$ generates terms such as  $f_u$ which in a way represents the contribution or effect of $\bsx_j$ for $j \in u$ beyond what can be explained by the lower order effects. Formally, we can write,
\begin{align}\label{eq:anovau}
f_u(\bsx) = 
\int_{\cx^{-u}} 
\Bigl(f(\bsx) - \sum_{v\subsetneq u}f_v(\bsx)\Bigr)\rd \bsx_{-u}.
\end{align}
Note that $f_u$ only depends on $\bsx_u$ although it is defined on $\cx^\otos$. If $u = \emptyset$, we define
\begin{align*}
f_\emptyset(\bsx)=\int_{\cx^{1{:}s}} f(\bsx)\rd\bsx = \mu.
\end{align*}
Let us define variances $\sigma^2_u =\int_{\cx^{1{:}s}}f_u(\bsx)^2\rd\bsx$ for $|u|>0$
and $\sigma^2_\emptyset=0$.
A useful property of the ANOVA decomposition gives us $\sum_{|u|>0}\sigma^2_u=\sigma^2$
where $\sigma^2=\int_{\cx^\otos} (f(\bsx)-\mu)^2\rd\bsx$.
Moreover, from the definition of $f_{\otos}$, we have $f(\bsx)=\sum_{u\subseteq1{:}s}f_u(\bsx)$, wherein we follow the convention that a integral over $\cx^{-\otos}$ leaves the function unaltered.

\subsection{Multiresolution}
\label{sec:multiresolution}

To explain the multiresolution of $L^2(\cx^\otos)$, we start by considering a version of Haar
wavelets in base $b$ which is adapted to $\cx\subset\real^d$ using a recursive split $\bx$ of $\cx$ in base $b\ge2$. As before, for ease of readability, we follow the same notation as in \cite{basu2015scrambled}.

Following Definition \ref{cell_def_2}, recall that we denote cells at level $k$ of a split by $\cx_{(k,t)}$ for $0\le t<b^k$. Here we have dropped the subscript $j$ since we are currently dealing with a single domain. Note that these cells can be further split into cells at level $k+1$ using
\begin{align}
\label{cellDef}
\cx_{(k,t)} = \bigcup_{c=0}^{b-1} \cx_{(k,t,c)},
\quad \text{where}\quad \cx_{(k,t,c)} =\cx_{(k+1,bt+c)}.
\end{align}

To explain the multiresolution of $\cx$ in terms of $\bx$
we introduce the following functions. Let $\varphi(\bsx)=1$ for all $\bsx\in\cx$
and
\begin{equation}\label{eq:unibasis}
\begin{split}
\psi_{ktc} 
&=  b^{(k+1)/2}1_{\bsx\in \cx_{(k,t,c)}}
-b^{(k-1)/2}1_{\bsx\in \cx_{(k,t)}}\\
& \equiv
b^{(k-1)/2}\bigl( b N_{ktc}(\bsx)-W_{kt}(\bsx)\bigr),
\end{split}
\end{equation}
where $N_{ktc}$ and $W_{kt}$ are indicator functions of
the given narrow and wide cells respectively. In~\eqref{eq:unibasis}, we scale by $b^{(k-1)/2}$ to make the norm of $\psi_{ktc}$ independent of $k$. In fact, it is easy to see that $\int \psi_{ktc}^2(\bsx)\rd\bsx=(b-1)/b$.

Now consider any $f_1,f_2\in L^2(\cx)$. Let
$\langle f_1,f_2\rangle = \int_{\cx}f_1(\bsx) f_2(\bsx)\rd\bsx$ denote their inner product. Furthermore, let
$$
f_K(\bsx) 
= \langle f,\varphi\rangle\varphi(\bsx)
+\sum_{k=1}^K\sum_{t=0}^{b^k-1}\sum_{c=0}^{b-1}
\langle f,\psi_{ktc}\rangle \psi_{ktc}(\bsx).
$$

Note that if $\bsx$ belongs to only one cell level $K+1$, then $f_K(\bsx)$ is the average of $f$ over that cell. Now for any $f\in L^1(\real^d)$, Lebesgue's differentiation theorem states that local averages over  
sets $\cs_K$, that converge nicely to $\bsx$, satisfy  
$$
\lim_{K\to\infty} \frac{\int_{\cs_K} f(\bsx)\rd\bsx}{\vol(\cs_K)}
= f(\bsx),\quad\text{a.e.}
$$   
Hence, if $\bx$ is convergent, then $\lim_{K\to\infty}f_K(\bsx)=f(\bsx)$
holds almost everywhere. Thus, assuming a convergent split we get, 
\begin{align}\label{eq:tightframe}
f(\bsx) = 
\langle f,\varphi\rangle\varphi(\bsx)
+\sum_{k=1}^\infty\sum_{t=0}^{b^k-1}\sum_{c=0}^{b-1}
\langle f,\psi_{ktc}\rangle \psi_{ktc}(\bsx).
\end{align}

Now using tensor products we can extend~\eqref{eq:tightframe} to the multidimensional setup.
We begin by generalizing the notation to a multidimensional case. For $j\in1{:}s$, let $\cx^{(j)}\subset\real^d$ have
recursive split $\bx_j$ in base $b\ge2$. Let $\varphi_j$ and $\psi_{j(ktc)}$ be the basis functions with narrow
and wide cell indicators $N_{jktc}$ and $W_{jkt}$.
For $u\subseteq1{:}s$, let $\kappa\in\natu^{|u|}$ have
elements $k_j\ge0$ for $j\in u$. 
Similarly let $\tau$ have elements $t_j\in\ints_{b^{k_j}}$ 
and  $\gamma$ have elements $c_j\in\ints_{b}$ for $j\in u$.

Now for any $\bsx\in\cx^{1{:}s}$ we define,
\begin{align}\label{eq:multiresfunction}
\psi_{u\kappa\tau\gamma}(\bsx)
:= \prod_{j\in u}\psi_{jk_jt_jc_j}(\bsx_j)
\prod_{j\not\in u}\varphi_j(\bsx_j).
\end{align}
Using \eqref{eq:multiresfunction}, the multiresolution of $L^2(\cx^{1{:}s})$ is
\begin{align*}
f(\bsx) &= \sum_{u\subseteq1{:}s}
\sum_{\kappa\mid u}
\sum_{\tau\mid u,\kappa}
\sum_{\gamma\mid u}
\langle\psi_{u\kappa\tau\gamma},f\rangle 
\psi_{u\kappa\tau\gamma}(\bsx)\\ 
 &= \mu + \sum_{|u|>0}
\sum_{\kappa\mid u}
\sum_{\tau\mid u,\kappa}
\sum_{\gamma\mid u}
\langle\psi_{u\kappa\tau\gamma},f\rangle 
\psi_{u\kappa\tau\gamma}(\bsx),
\end{align*}
where sums are over their entire ranges given the other variables. 

\subsection{Variance and gain coefficients}
\label{sec:var_and_gain}
Here we study the variance of the scrambled
geometric net estimator. Let $\{\bsa_i\}_{i=1}^n \in[0,1]^s$ be an arbitrary set of $n$ points not necessarily a digital net. Let $\bsu_i\in[0,1]^s$ for $i = 1,\ldots, n$ be its nested uniform scramble. We then map it to
$\bsx_i\in\cx^{1{:}s}$ using recursive splits in base $b$.

Using ideas from \cite{owensinum} we have,
\begin{equation*}
\var(\hat{\mu}) = \sum_{|u| > 0} \sum_{\kappa\mid u} 
\var\biggl( \frac{1}{n} \sum_{i=1}^n \nu_{u\kappa} (\bsx_i) \biggr),
\end{equation*}
where
\begin{equation}
\label{eq:nu_uk}
\nu_{u\kappa}(\bsx) = \sum_{\tau\mid u,\kappa} \sum_{\gamma\mid u} \langle f, \psi_{u\kappa\tau\gamma} \rangle \psi_{u\kappa\tau\gamma}(\bsx),
\end{equation}
with $\nu_{\emptyset,()} = \mu$. Note that the function $\nu_{u\kappa}$ is constant within regions of the form
\[
\prod_{j \in u} \cx_{j,(k_j,t_j,c_j)}\prod_{j \not\in u} \cx^{(j)},
\]
for $0 \le t_j < b^{k_j }$ and $0\le c_j<b$. Further, define
\[ \sigma_{u\kappa}^2 := \int_{\cx^\otos} \nu_{u\kappa}^2(x) \rd \bsx.\]
Now using the multiresolution-based ANOVA decomposition, we have
\begin{align}\label{eq:mranova}
\sigma^2 = \int_{\cx^{1{:}s}} (f(\bsx) - \mu)^2 \rd \bsx 
= \sum_{|u| > 0}\sum_{\kappa\mid u} \sigma_{u\kappa}^2.
\end{align}

If we assume that $\bsa_1, \ldots, \bsa_n$ is a $(t,m,s)$-net in base $b$, then its equidistribution property
determines the contribution of each $\nu_{u\kappa}$ to $\var(\hat{\mu})$. 
Let $\bsa_i=(a_{i1},\dots,a_{is})$ and define
\[ \Upsilon_{i,i',j,k} := \frac{1}{b-1} \left(b1_{\floor{b^{k+1}a_{ij}} = \floor{b^{k+1}a_{i'j}}}- 1_{\floor{b^{k}a_{ij}} = \floor{b^{k}a_{i'j}}} \right).
\]
Furthermore, for each $|u| > 0$ and $\kappa\in\natu^{|u|}$ define the gain coefficients as
\begin{align}
\label{gain_coeff}
\Gamma_{u,\kappa} := \frac{1}{n}  \sum_{i=1}^n \sum_{i'=1}^n \prod_{j \in u} \Upsilon_{i,i',j,k_j}.
\end{align}
Now using Theorem 2 of \cite{owensinum} we have, 
\begin{align}
\label{var_form_original}
\var(\hat{\mu}) = \frac{1}{n} \sum_{|u| >0} \sum_{\kappa\mid u} \Gamma_{u,\kappa} \sigma_{u\kappa}^2.
\end{align}
In usual Monte Carlo sampling, $\var(\hat \mu) = {\sigma^2}/n$ which
corresponds to all $\Gamma_{u,\kappa}=1$.
The $\Gamma_{u,\kappa}$ are
called gain coefficients since they portray the relative gain compared to usual Monte Carlo. If the point sets $\bsa_i$ are carefully selected then the gain coefficients can be reduced and we can get a substantial improvement over the usual Monte Carlo. The upper bound on $\Gamma_{u,\kappa}$ is obtained in \cite{snxs}. Assuming $\bsa_1, \ldots, \bsa_n$ being a $(t,m,s)$-net and using the upper bound on gain coefficients, the authors in \cite{basu2015scrambled} show that for certain smooth functions $f$ on $\cx^\otos$
\begin{align*}
\var(\hat{\mu}) \le C \frac{(\log n)^{s-1}}{n^{1+ 2/d}}.
\end{align*}

\section{Proof of Theorem \ref{thm:lower-bound}}
\label{section:theorem1}
The proof the theorem is organized as follows. First we state a sequence of lemmas essential for the construction of the proof. Thereafter assuming the validity of these lemmas we finish the proof of Theorem 1. The proofs of these lemmas are deferred to the appendices. Finally we note that the constants appearing in the proof of this theorem are generic and are allowed to change from one line to other, without compromising the validity of the arguments. Throughout our proofs, $f(x) = O(g(x))$ implies, $|f(x)| \le M|g(x)|$ for all $x \ge x_0$ and some positive constant $M$. 


\subsection{Supporting Lemmas}
We begin with a sequence of lemmas. The first lemma gives a lower bound on the gain coefficients $\Gamma_{u,\kappa}$ as defined in \eqref{gain_coeff}. 

\begin{lemma}  
 \label{lem:lower_bound_gamma}
 Let $b \ge \max(s,2)$. Under the above notation,
	\begin{align*}
	\Gamma_{u,\kappa} \ge \left(\frac{b}{b-1}\right)^{\min(m, s-2)}\left(1 - \frac{\min(m, s-2)}{b-1}\right) =: c_{g},
	\end{align*}
	if $|\kappa| > m - |u|$, and $\Gamma_{u,\kappa} = 0$ otherwise.	
\end{lemma}

We refer to Appendix \ref{app_lb_gain} for the proof of Lemma \ref{lem:lower_bound_gamma}. The proofs of the remaining sequence of lemmas are presented in Appendix \ref{app_lb_sig}.

\begin{lemma}
	\label{rule1}
	Let $f : \cx \rightarrow \mathbb{R}$ such that $\norm{\nabla f(\bsx) - \nabla f(\bsx^*)} \le B \norm{ \bsx - \bsx^*}^\beta$ for some finite $B \ge 0$ and $\beta \in (0,1]$ for all $\bsx, \bsx^* \in \cx$. Then,
	\[f(\bsx) = f(\bsx^*) + \langle \nabla f(\bsx^*), \bsx - \bsx^*\rangle + C \norm{ \bsx - \bsx^*}^{1+\beta}
	\]
	where $|C| \le B(1+ \beta)^{-1} \le B$ and $\norm{ \cdot } $ is the Euclidian norm.
\end{lemma}

\begin{lemma}
	\label{rule2}
	Let $f$ be as in Lemma \ref{rule1}. Then,
	\begin{align*}
	b^{k+1}\int_{\cx_{(k,t,c)}} f(\bsx) \rd \bsx &- b^{k}\int_{\cx_{(k,t)}} f(\bsx) \rd \bsx \\&= \langle \bsn_{ktc} - \bsw_{kt}, \nabla f(\bsw_{kt})\rangle  
	+ O(b^{-k(1+\beta)/d}),
	\end{align*}
	where $\bsn_{ktc}, \bsw_{kt}$ denotes the center of $\cx_{(k,t,c)},\cx_{(k,t)}$ respectively.
\end{lemma}

The following lemma about the smoothness properties of the ANOVA components of $f$ will be crucial in our analysis of asymptotic properties of $W$.

\begin{lemma}
	\label{lem:fu_smooth}
	Let $u \subseteq \otos$ and let $f_u$ denotes the ANOVA component of $f$ as defined in \eqref{eq:anovau}. If $f$ is smooth on $\cx^\otos$, then $f_u$ is also smooth.
\end{lemma}

To state the next lemma, we begin by introducing a few notations. Let $u \subseteq \{1, \ldots, s\}$. Let $\kappa$, $\tau$, and $\gamma$ be $|u|$-tuples with components $k_j \in \mathbb{N}$, $t_j \in \mathbb{Z}_{b^{k_j}}$ and $c_j \in \mathbb{Z}_b$ respectively for $j \in u$. Let $\psi_{u\kappa\tau\gamma}$ be the multiresolution basis function defined in \eqref{eq:multiresfunction}. To simplify notation let us define a set of multi-indices $\eta$ of length $|u|$ as $S_u := \{1, \ldots, d\}^{|u|}$.
Further, for any $\eta \in S_u$ define the mixed partial $\partial^\eta f $ as 
\begin{align*}
\partial^\eta f (\bsx) = \frac{\partial^{|u|}f (\bsx) }{\prod_{j \in u} \partial x_{j,\eta_j}}.
\end{align*}

\begin{lemma}
	\label{lem:inn_prod_s}
	Let $f \in \mathcal{F}_s$. Under the above notation,
	\begin{align*}
	\langle f, \psi_{u\kappa\tau\gamma} \rangle = b^{-(|\kappa| + |u|)/2} \sum_{ \eta \in S_u} \left(\prod_{j \in u}(\bsn_{c_j} - \bsw_j)_{\eta_j}\right) \partial^{\eta} f_u( \bsw) + O\left(b^{-\frac{|\kappa|}{2}\left(1+\frac{2}{d}\right) - \frac{\tilde{k}\beta}{d}}\right)
	\end{align*}
	where $\bsn_{c_j}, \bsw_j$ defined in \eqref{center_property}, $\tilde{k} = \min_{j \in u} k_j$ and $\bsw = \{\bsw_j : j \in u\} $.
\end{lemma}

\begin{lemma}
	\label{sigmauk}
	Let $f \in \mathcal{F}_s$. Under the above notation,
	\begin{align*}
	\sigma_{u,\kappa}^2 &=  b^{-(|\kappa|+|u|)} \sum_{\tau} (\nabla^u f_u(\bsw_{\tau}))^T \tilde{A}_{u} \nabla^{u} f_u (\bsw_{\tau}) + O(b^{-2|\kappa|/d - \tilde{k}\beta/d})
	\end{align*}
	where $\tilde{A}_u = \bigotimes_{j \in u}A_j$ and $A_j$ is defined in \eqref{eq:Aj}.
\end{lemma}

\subsection{Completing Proof of Theorem \ref{thm:lower-bound}}

The main idea of the proof of the lower bound is along the following lines. 
Following Lemma \ref{lem:lower_bound_gamma} and using \eqref{var_form_original} we have,
\begin{align}
\label{var_form}
\var(\hat{\mu}) \ge \frac{c_{g}}{n} \sum_{|u| > 0}\sum_{|\kappa | > m - |u|}  \sigma_{u,\kappa}^2. 
\end{align}
Note that from Lemma \ref{sigmauk} we see,
\begin{align}
\label{eq:sigma_expansion}
\sigma_{u,\kappa}^2 &=  b^{-(|\kappa|+|u|)} \sum_{\tau} (\nabla^u f_u(\bsw_{\tau}))^T \tilde{A}_{u} \nabla^{u} f_u (\bsw_{\tau}) + O(b^{-2|\kappa|/d - \tilde{k}\beta/d}).
\end{align}
Since $\lambda_1(A_j) \ge \tilde{c} b^{-2k_j/d}$ for all $j = 1, \ldots, s$, we get $\lambda_1(\tilde{A}_u) \ge \tilde{c}^{|u|}b^{-2|\kappa|/d}$ for all $u$. Using this, we have,
\begin{align*}
 b^{-(|\kappa|+|u|)} \sum_{\tau} (\nabla^u f_u(\bsw_{\tau}))^T \tilde{A}_{u} \nabla^{u} f_u(\bsw_{\tau}) &\ge  b^{-(|\kappa| + |u|)}\sum_{\tau} \norm{\nabla^u f_u(\bsw_{\tau})}^2 \lambda_1(A_u)\\
&\ge c b^{-|\kappa|} \sum_{\tau} b^{-2|\kappa|/d} \norm{\nabla^u f_u(\bsw_{\tau})}^2.
\end{align*}
Then by Riemann integrability, there exists a $K:= K(f)$ such that for all $|\kappa| \ge K$, 
\begin{align*}
\sum_{\tau} b^{-|\kappa|} \norm{\nabla^u f_u(\bsw_{\tau})}^2 \ge \frac{1}{2} \int_{\cx^u} \norm{\nabla^u f_u (\bsx_u)}^2 \rd \bsx_u.
\end{align*}
Now using \eqref{eq:anovau} and the Leibniz integral rule we have,
\begin{align*}
\nabla^u f_u (\bsx_u) = \int_{\cx^{-u}} \nabla^u f(\bsx) \rd \bsx_{-u}.
\end{align*}
Denoting $ \nabla^u f(\bsx) = (v_1(\bsx), \ldots, v_{d^{|u|}}(\bsx))$, we have by Cauchy-Schwarz inequality and the definition of $\cf_s$, 
\begin{align*}
\int_{\cx^u} \norm{\int_{\cx^{-u}} \nabla^u f(\bsx) \rd \bsx_{-u}}^2 \rd \bsx_u &= \int_{\cx^u}\sum_{j = 1}^{d^{|u|}} \left( \int_{\cx^{-u}} v_j(\bsx) \rd \bsx_{-u}\right)^2 \rd \bsx_u\\  
&\geq \sum_{j = 1}^{d^{|u|}} \left( \int_{\cx^u}  \int_{\cx^{-u}} v_j(\bsx) \rd \bsx_{-u} \bsx_u \right)^2 \\
& =  \norm{ \int_{\cx^\otos} \nabla^u f (\bsx) \rd \bsx }^2  > 0.
\end{align*}

Hence we get,
\begin{align}
\label{eq:first_term_bound}
 b^{-(|\kappa|+|u|)} \sum_{\tau} (\nabla^u f_u(\bsw_{\tau}))^T \tilde{A}_{u} \nabla^{u} f_u (\bsw_{\tau}) \ge cb^{-2|\kappa|/d}.
\end{align}

Now let $\{a_m\}$ be a diverging sequence (to be decided later) such that $\tilde{k} \ge a_m$. Since $m \rightarrow \infty$, for large enough $m$ and $ |\kappa| > m - |u|$, using \eqref{eq:sigma_expansion} and \eqref{eq:first_term_bound}, we have $\sigma_{u, \kappa}^2 \ge c b^{-2|\kappa|/d}$. Therefore using  \eqref{var_form} we have,
\begin{align*}
\var(\hat{\mu}) &\ge \frac{c}{n} \sum_{|u| > 0}\sum_{\substack{|\kappa | > m - |u| \\ \tilde{k} \ge a_m}}  \sigma_{u,\kappa}^2 \ge \frac{c}{n} \sum_{|u| > 0}\sum_{\substack{|\kappa | > m - |u| \\ \tilde{k} \ge a_m}} b^{-2|\kappa|/d}
\end{align*}
Since we are interested in the limit as $m \rightarrow \infty$, we can assume that $m \ge s$. For such large $m$, we have
\begin{align*}
\sum_{\substack{|\kappa | > m - |u| \\ \tilde{k} \ge a_m}} b^{-2|\kappa|/d} = \sum_{r = m - |u| + 1}^{\infty} b^{-2r/d} {r - a_m|u| + |u| - 1 \choose  |u| - 1} 
\end{align*}
where the binomial coefficient is the number of $|u|$-vectors $\kappa$ of non negative integers that sum to $r$ and individually are greater than or equal to $a_m$. Making the substitution $l = r - m + |u|$ we have,
\begin{align*}
\sum_{\substack{|\kappa | > m - |u| \\ \tilde{k} \ge a_m}} b^{-2|\kappa|/d} &= b^{-2(m - |u|)/d} \sum_{l=1}^{\infty} b^{-2l/d} { l + m - a_m|u| - 1 \choose |u| - 1} \\
&\ge \frac{b^{-2(m - |u|)/d}}{(|u| - 1)!} \sum_{l=1}^\infty b^{-2l/d} (l+ m - a_m|u|  -|u| +1)^{|u| - 1}\\
&= \frac{b^{-2(m - |u|)/d}}{(|u| - 1)!} \sum_{l=1}^\infty b^{-2l/d} \sum_{j=0}^{|u| - 1} {|u| - 1 \choose j} l^j (m - a_m|u|  - |u| + 1)^{|u| - 1- j}\\
&= b^{-2(m - |u|)/d}  \sum_{j=0}^{|u| - 1} \frac{(m - a_m|u| - |u| + 1)^{|u| - 1- j}}{j! (|u| - 1-j)!} \sum_{l=1}^\infty b^{-2l/d} l^j \\
& \ge b^{-2(m - |u|)/d} \frac{(m - a_m|u|  - |u| + 1)^{|u| - 1}}{(|u| - 1)!}  \sum_{l=1}^\infty b^{-2l/d} \\
&\ge \frac{c}{n^{2/d}} (m - a_m|u| )^{|u| - 1}.
\end{align*}
Now $m = \log_b(n)$ and $|u| \le s$. Choosing $a_m$ to be diverging slowly enough to guarantee that 
$m - a_ms \ge m/2$, we have
\begin{align*}
\var(\hat{\mu}) \ge \frac{c}{n^{1+2/d}} \sum_{|u| > 0} m^{|u| - 1} \ge c\frac{m^{s-1}}{n^{1+2/d}}. 
\end{align*}
This completes the proof of Theorem \ref{thm:lower-bound}. 
$\hfill\square$

\section{Specific domains}
\label{sec:proof_cor}
In this section, we show that the conditions of Theorem \ref{thm:lower-bound} hold when the domain is a triangle or the unit interval. We further show that if we choose a different class of functions defined on $\cx = T^2$ and $s = 1$, the same lower bound \eqref{eq:sigma_sgn} holds.

\subsection{Proof of Corollary \ref{cor:triangle}}
To prove the result for a triangle, we show the explicit form of the matrix $A_{j}$ defined in \eqref{eq:Aj} by using the fact that $\cx = T^2$. Using the notation from Figure \ref{fig:subdiv} we have,
\begin{align*}
\bsn_{c_j} - \bsw_{j} &= \begin{cases}
(-r_j/6,-r_j/6) & \text{ for } c_j = 0\\
(r_j/3,-r_j/6) & \text{ for } c_j = 1\\
(-r_j/6,r_j/3) & \text{ for } c_j = 2\\
(0,0) & \text{ for } c_j = 3
\end{cases}
\qquad \text{ for } t_j \text{ such that } \cx_{j,(k_j,t_j)} \text{ is upright}\\
\bsn_{c_j} - \bsw_{j} &= \begin{cases}
(r_j/6,r_j/6) & \text{ for } c_j = 0\\
(-r_j/3,r_j/6) & \text{ for } c_j = 1\\
(r_j/6,-r_j/3) & \text{ for } c_j = 2\\
(0,0) & \text{ for } c_j = 3
\end{cases}
\qquad \text{ for } t_j \text{ such that } \cx_{j,(k_j,t_j)} \text{ is inverted}
\end{align*}
where $r_j^2 = 2 b^{-k_j}$. Using this we get,
\begin{align*}
A_{j} =  \frac{r_j^2}{6} \begin{bmatrix}
1 & -1/2\\
-1/2 & 1
\end{bmatrix} = \frac{b^{-k_j}}{6} \begin{bmatrix}
2 & -1\\
-1 & 2
\end{bmatrix}. 
\end{align*}
Thus, $\lambda_1(A_j) = b^{-k_j}/6$. Hence the lower bound follows from Theorem \ref{thm:lower-bound}. The upper bound follows from \cite{basu2015scrambled}. $\hfill \square$

\subsection{Proof of Corollary \ref{cor:square}} Following the proof for $T^2$ it is now easy to see that if $\cx = [0,1]$, then
\begin{align*}
A_j = b^{-2k_j} \left(\frac{b^2 - 1}{12b}\right).
\end{align*}
Now using the same argument as in Corollary \ref{cor:triangle} we get the desired result.



\subsection{Alternative approach for $ s = 1, \cx=T^2$}
\label{sec:specific_domain}


Here we give a different proof for \eqref{eq:sigma_sgn} assuming $\cx = T^2$. Since $s = 1$, $\tau, \kappa$ and $\gamma$ are one dimensional quantities which we denote it by $t,k$ and $c$ respectively.  We consider a different class of functions $\cg$ as follows. 

\begin{definition}
Let $\cg$ be the collection of functions $f : T^2 \rightarrow \real$ such that $\partial^{1:2} f$ is continuous and for some $\tilde{c} > 0$ either $\partial^{\{1\}} f (\bsx) > \tilde{c}$ or $\partial^{\{2\}}f(\bsx) > \tilde{c}$ for all $\bsx \in T^2$. 
\end{definition}

Now we prove the required lower bound on $\var(\hat{\mu})$ for $f \in \cg$. 

\begin{propo}
\label{prop:triangle}
Let $f \in \cg$. Then under the above notation, 
$$\var(\hat{\mu}) \ge \frac{c}{n^2}.$$
for some $c > 0$.
\end{propo}

\begin{proof}
Crucial to the proof of Proposition \ref{prop:triangle} is the following lemma, proof of which can be found in Appendix \ref{app:lower_bound_triangles}.
\begin{lemma} 
\label{lem:sig_form}
Under the above notation we have,
\begin{align}
\label{eq:sig_form}
\sigma_{k}^2 =  \sum_{t = 0}^{b^k - 1} \sum_{\ell = 1}^{b-1} \frac{b^{k + 1}}{\ell(\ell+1)} \left[ \sum_{i = 1}^\ell \left( \int_{\cx_{(k,t,i-1)}} f(x) dx - \int_{\cx_{(k,t, \ell) }} f(x) dx \right)\right]^2.
\end{align}
\end{lemma}

Assuming the validity of Lemma \ref{lem:sig_form}, we continue the proof of Proposition \ref{prop:triangle}. 
Note that the constants appearing in the proof of this proposition are generic and are allowed to change from one line to other, without compromising the validity of the arguments.

Consider the splitting of the triangle in base $b = 4$ introduced in \cite{Basu2014} to give an explicit form to $\cx_{(k,t,c)}$. To give a lower bound on $\sigma_{k}^2$, we need to consider $I_i$, for $i = 1,2,3$, where 
\begin{align*}
I_1 &:= \left(\int_{\cx_{(k,t,0)}} f(\bsx) d\bsx - \int_{\cx_{(k,t,1)}} f(\bsx) d\bsx\right)^2,\\
I_2 &:= \left(\int_{\cx_{(k,t,0)}} f(\bsx) d\bsx + \int_{\cx_{(k,t,1)}} f(\bsx) d\bsx -  2\int_{\cx_{(k,t,2)}} f(\bsx) d\bsx \right)^2,\\
I_3 &:= \left(\int_{\cx_{(k,t,0)}} f(\bsx) d\bsx + \int_{\cx_{(k,t,1)}} f(\bsx) d\bsx +  \int_{\cx_{(k,t,2)}} f(\bsx) d\bsx - 3\int_{\cx_{(k,t,3)}} f(\bsx) d\bsx \right)^2.
\end{align*}

Fix any $k$ and $t$ and consider the splitting of $T^2$ as given in Figure \ref{fig:labeling}. At the level $k$ triangulation, denote the length of the orthogonal sides of cell $t$ by $r$. It is easy to see that $r = \sqrt{2}b^{-k/2}$.  If $t$ denotes an upright sub-triangle, its co-ordinates can be written as $(\alpha r, \beta r)$, $((\alpha + 1)r, \beta r)$ and $(\alpha r, (\beta + 1) r)$, where $\alpha, \beta \in \{0, 1, \ldots, 2^{k} -1\}$.  Similarly, if $t$ denotes an inverted triangle, then following the figure, the coordinates can be written as $(\alpha r, (\beta+1)r), ((\alpha+1)r, (\beta+1)r)$ and $((\alpha+1)r, \beta r)$. Note that since we only need $\cx_{(k,t,c)}$ we zoom into $\cx_{(k,t)}$ to identify $\cx_{(k,t,c)}$. 

\begin{figure}[!t]
	\includegraphics[scale=0.7]{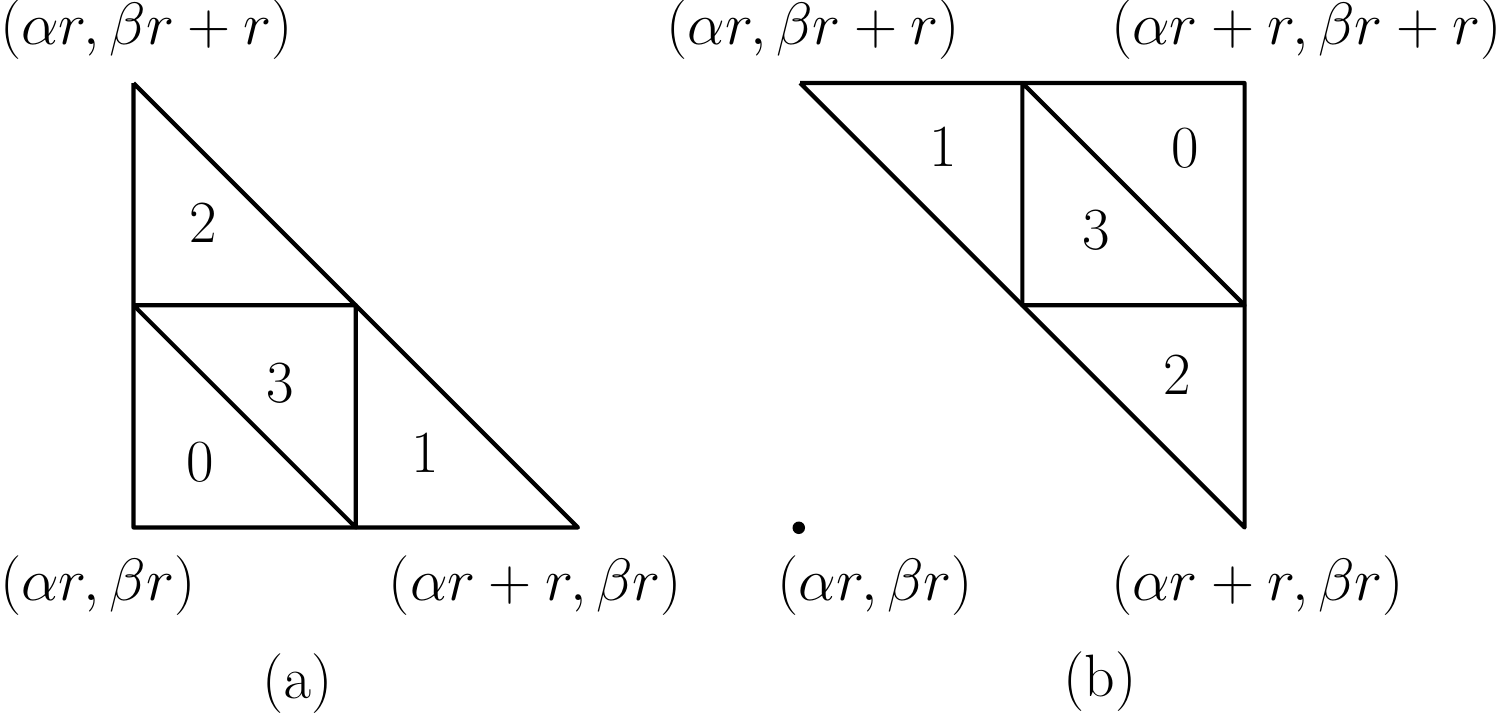}
	\caption{\label{fig:labeling}The labelling of each cell $\cx_{(k,t,c)}$ for a level $k$ triangulation of $T^2$. Subfigure (a) and (b) denotes the upright  and inverted cases of $\cx_{(k,t)}$ respectively.}
\end{figure}

We now give a lower bound to $\sigma_{k}^2$ for $f \in \mathcal{G}$.
Without the loss of generality we assume $\partial^{\{1\}}f(\bsx) > \tilde{c}$ for all $\bsx \in T^2$. Using \eqref{eq:sig_form} we have,
\begin{align}
	\label{eq:sig_final}
	\sigma_{k}^2 &=  \sum_{t = 0}^{b^k - 1} \sum_{\ell = 1}^{b-1} \frac{b^{k + 1}}{\ell(\ell+1)} \left[ \sum_{i = 1}^\ell \left( \int_{\cx_{(k,t,i-1)}} f(x) dx - \int_{\cx_{(k,t,\ell)}} f(x) dx \right)\right]^2 \nonumber \\
	&= b^{k + 1} \sum_{t = 0}^{b^k - 1} \left(\frac{I_1}{2} + \frac{I_2}{6} + \frac{I_3}{12}\right) \nonumber \\
	&\ge  \frac{b^{k+1}}{2} \sum_{t = 0}^{b^k - 1} I_1 = \frac{ b^{k+1} }{2} \left[\sum_{t : \cx_{(k,t)} \text{is upright}}I_1+ \sum_{t : \cx_{(k,t)} \text{is inverted}} I_1 \right] \nonumber \\
	&\ge \frac{ b^{k+1} }{2} \sum_{\alpha = 0}^{2^k -1} \sum_{\beta = 0}^{2^k -1 -\alpha}I_1^{\text{upright}} 
	\end{align}
where the last line follows by keeping the terms corresponding to $\cx_{(k,t)}$ being upright.  Now it is enough to give a lower bound on $I_1^{\text{upright}}$. Fix any $\epsilon > 0$ and let $P_\epsilon^0$ be a partition of $\cx_{(k,t,0)}$. By congruency, there exists a partition $P_{\epsilon}^1$ on $\cx_{(k,t,1)}$ such that if $\mathcal{C} \in P_{\epsilon}^0$ then $\mathcal{C} + r/2 \in P_{\epsilon}^1$. Using the notation from Figure \ref{fig:labeling}, we have 
	\begin{align*}
	\int_{\cx_{(k,t,0)}} f(\bsx) d\bsx &\le \sum_{\cc_i \in P_\epsilon^0}|\mathcal{C}_i| f(x_{1i}, x_{2i}) + \epsilon\\
	\int_{\cx_{(k,t,1)}} f(\bsx) d\bsx &\ge  \sum_{\cc_i \in P_\epsilon^0}|\mathcal{C}_i| f(x_{1i} + r/2, x_{2i})  - \epsilon
	\end{align*}
	Combining the above equations and using mean value theorem we have for all $\epsilon > 0$,
	\begin{align*}
\int_{\cx_{(k,t,1)}} f(\bsx) d\bsx &-\int_{\cx_{(k,t,0)}} f(\bsx) d\bsx \\ &\ge  \sum_{C_i \in P_\epsilon^0}|\mathcal{C}_i| \left( f(x_{1i} + r/2, x_{2i}) - f(x_{1i},x_{2i})\right) - 2\epsilon \\
& = \sum_{C_i \in P_\epsilon^0}|\mathcal{C}_i| \left( \frac{r}{2} \frac{\partial f}{\partial x_1}(\xi(x_{1i}),x_{2i}) \right) - 2\epsilon \\
&\ge \tilde{c} \frac{r^3}{16} - 2\epsilon,
\end{align*}
where $\xi(x_{1i}) \in (x_{1i}, x_{1i} + r/2)$. Therefore, we get for some $c > 0$,
\begin{align*}
\int_{\cx_{(k,t,1)}} f(\bsx) d\bsx &-\int_{\cx_{(k,t,0)}} f(\bsx) d\bsx  \ge c r^3.
\end{align*}
Thus, we have
\begin{align*}
\sigma_{k}^2 & \ge \frac{ b^{k+1} }{2} \sum_{\alpha = 0}^{2^k -1} \sum_{\beta = 0}^{2^k -1 -\alpha}I_1^{\text{upright}} \ge c b^{2k} r^6 \ge cb^{-k}.
\end{align*}
The proof now follows from \eqref{var_form}.
\end{proof}


\section{Stein's Method and Sketch of Proof of Theorem \ref{thm:convergence}} \label{section:theorem2}
Our proof closely follows the proof of Theorem 3 in Loh \cite{loh:2003} obtained for $[0,1]^s$. However, for the sake of completeness, we briefly describe the proof for $\cx^s$ by appropriately changing the arguments. In particular the technique relies on Stein's Method of exchangeable pairs for proving asymptotic normality of a sequence of random variables. The idea of Stein's Method can be described as follows. 

To demonstrate asymptotic normality for the sequence $\{T_m\}_{m \ge 1}$, it is enough to show that for $Z\sim N(0,1)$, $\sup_{g \in \mathcal{G}}|\E(g(T_m))-\E(g(Z))|\rightarrow 0$ as $m \rightarrow \infty$  for a suitable class of test functions $\mathcal{G}$. Stein's Method relies on obtaining suitable bounds on this quantity by using the characteristics of a standard normal distribution. In particular, a random variable is $Z$ has a standard normal distribution if $\E(h^{'}(Z)-Zh(Z))=0$ for a large enough class of ``nice" functions $h$, where $h^{'}$ denotes the derivative of $h$.  Therefore, if the distribution of a random variable $T_m$ is asymptotically close to standard normal, then one expects that  for large enough $m$,  $|\E(h^{'}(T_m)-T_mh(T_m))|$ to be small for suitable class of functions $h$. This motivates defining the Stein's Equation, namely $g(w)-\E(g(Z))=h^{'}(w)-wh(w)$. Consequently, since any solution $h_g$ to this equation satisfies $|\E(g(T_m))-\E(g(Z))|=|\E ( h_g^{'}(T_m)-T_mh_g(T_m))|$, one has $\sup_{g \in \mathcal{G}}|\E(g(T_m))-\E(g(Z))|\le \sup_{h \in \mathcal{H}}|\E ( h^{'}(T_m)-T_mh(T_m)) |$ for any $\mathcal{H}\supseteq \{h_g: g\in \mathcal{G}\}$. In particular, obtaining Berry-Esseen type bounds on the convergence of $T_m$ to normality requires $\mathcal{G}=\{g_t(\cdot):=\mathds{1}(\cdot\le t): t\in \mathbb{R}\}$ where $\one(\cdot)$ is the indicator function, and the following Lemma (\citep{barbour2005introduction, stein1972}) is crucial in bounding $\sup_{h \in \mathcal{H}}|\E ( h_{g_t}^{'}(T_m)-T_mh_{g_t}(T_m))|$ for each $t \in \mathbb{R}$.


\begin{lemma}
\label{lem:stein}
Let $\Phi$ and $\phi$ denote the cumulative distribution function and probability density function of the standard normal distribution, respectively. For every $t \in \mathbb{R}$, the unique bounded solution $h_t : \mathbb{R} \rightarrow \real$ of the differential equation 
\begin{align*}
h^{'}(w) - w h(w) = \mathds{1}(w \le t) - \Phi(t) \qquad \forall\; w\in \real
\end{align*}
is given by
\begin{align*}
h_t(w) = \begin{cases}
\Phi(w)(1 - \Phi(t))/\phi(w), & \text{ if } w \le t,\\
\Phi(t)(1 - \Phi(w))/\phi(w), & \text{ if } w > t.\\ 
\end{cases}
\end{align*}
Furthermore, $0 \le h_t(w) \le 1$ and $|h_t^{'}(w)| \le 1$ for all $w \in \real$.
\end{lemma}
According to the above result, one has by an application of Slutsky's Lemma, that to show $W$ is asymptotically normal as $m \rightarrow \infty$, it is enough for us to uniformly bound $|\E (h_{t}^{'}(\tilde{W})-\tilde{W}h_{t}(\tilde{W}))|$ for $h_t$ as in Lemma \ref{lem:stein} and any $\tilde{W}$ satisfying $W-\tilde{W}\rightarrow 0$ in probability for $W$ defined in \ref{Wdef}. One such convenient $\tilde{W}$ can be introduced as follows. Rewriting $W$ by using the multiresolution analysis as in introduced in Section \ref{sec:multiresolution}, we note that,
\begin{align*}
W &= \frac{\hat{\mu} - \mu}{\sigma_{sgn}}
=  \frac{1}{n\sigma_{sgn}} \sum_{i=1}^{n} (f(\bsx_i) - \mu)
=  \frac{1}{n\sigma_{sgn}} \sum_{i=1}^{n} \sum_{|u|>0}
\sum_{\kappa\mid u} \nu_{u\kappa}(\bsx_i),
\end{align*}
where $ \nu_{u\kappa}(\cdot) $ is defined in \eqref{eq:nu_uk}. Note that $u$ chooses a subset of $\{1,\ldots,s\}$ and $\kappa$ is an $|u|$-dimensional vector containing the levels of partitions. To make the notation simpler we introduce $\tilde{\bsk}$ as a $s$-dimensional vector where $\tilde{\bsk}_u = \kappa$ and $\tilde{\bsk}_{-u} = 0$. Now suppose that if $\bsa, \bsb \in \real^s$ then,
\begin{enumerate}[(i)]
\item $\bsa \preceq \bsb$ if and only if $a_j \le b_j$ for all $1 \le j \le s$;
\item  $\bsa \prec \bsb$ if and only if $a_j \le b_j$ for all $1 \le j \le s$ with at least one strict inequality. 
\end{enumerate}
Using this notation we get,
\begin{align}
\label{eq:multi_form_W}
W &= \frac{1}{n\sigma_{sgn}} \sum_{i=1}^{n} \sum_{\tilde{\bsk} : \bzero \prec \tilde{\bsk}} \nu_{\tilde{\bsk}}(\bsx_i) = \frac{1}{n\sigma_{sgn}} \sum_{i=1}^{n} \sum_{\tilde{\bsk} \succeq \bzero, |\tilde{\bsk}| \ge m+1} \nu_{\tilde{\bsk}}(\bsx_i),
\end{align}
where the last equality follows from a consequence of the ANOVA decomposition and the definition of a $(0,m,s)$-scrambled geometric net.
Finally we define 
\begin{align}
\label{eq:Wtilde}
\tilde{W} =  \frac{1}{n\sigma_{sgn}} \sum_{i=1}^{n} \sum_{\tilde{\bsk} \succeq \tilde{m}\one , |\tilde{\bsk}| \ge m+1} \nu_{\tilde{\bsk}}(\bsx_i).
\end{align}
for $\tilde{m} = \floor{2s \log_b{m}}$. The next Lemma guarantees $W-\tilde{W}\rightarrow 0$ in probability and can be proved along the lines of Proposition 1 in Loh \citep{loh:2003}.
\begin{lemma}
Let $b \ge \max \{s,d,2\}$ and $f \in \mathcal{F}_s$. Then $\E( W - \tilde{W})^2 = O(\tilde{m}/m)$, where $W$ and $\tilde{W}$ are as in \eqref{eq:multi_form_W} and \eqref{eq:Wtilde} respectively. 
\end{lemma}


Therefore, it suffices to show that under the assumptions of Theorem \ref{thm:convergence}, 
\begin{align*}
\sup \left\{ |\P(\tilde{W} \le w) - \Phi(w) | : -\infty < w < \infty \right\} = O \left( \left(\frac{\log_b(m)}{m}\right)^{1/2}\right).
\end{align*}
which as argued earlier can be achieved by suitably bounding $|\E(h_{t}^{'}(\tilde{W})-\tilde{W}h_{t}(\tilde{W}))|$ for $h_t$ as in Lemma \ref{lem:stein}. The necessary control over $|\E (h_{t}^{'}(\tilde{W})-\tilde{W}h_{t}(\tilde{W}))|$ will be obtained by the exchangeable pair technique of Stein's method.  This is done as follows.


 Pick $(I,J)$ uniformly from $\{1, \ldots, n\}\times\{1, \ldots, s\}$. Furthermore let,
\[
\{ \pi_{j\sumdot}^*, \pi_{j\sumdot a_{1}}^*, \ldots ,\pi_{j\sumdot a_{1}a_{2},\dots,a_{k}}^*, \ldots : 1\le j \le s,\;\; 0 \le a_k \le b-1, \;\;k = 1,2,\ldots 
\}
\]
be an independent replication of the array of $\pi$'s as introduced in Section \ref{sec:background}. In particular, these draws of $I,J, \pi^*$'s are made independently of each other as well as of all previously defined random variable. Now for $1 \le i \le n, 1 \le j \le s$, define
\begin{align*}
&\tilde{\pi}_{j\sumdot a_{i,j,1},\ldots, a_{i,j,k_j-1}} \\
&= 
\begin{cases}
\pi_{j\sumdot a_{i,j,1},\ldots, a_{i,j,k_j-1}}^*, & \text{if } J = j, \;\; k_j \ge \tilde{m} \text{ and }\\
& (a_{I,j,1}, \ldots , a_{I,j,\tilde{m} - 1}) = (a_{i,j,1}, \ldots , a_{i,j,\tilde{m} - 1}), \\
\pi_{j\sumdot a_{i,j,1},\ldots, a_{i,j,k_j-1}} & \text{otherwise},
\end{cases}
\end{align*}
where $a_{i,j,k}$'s are the corresponding bits for the $(0, m, s)$-net $\{\bsa_i : i = 1, \ldots, n\}$. Let $\tilde{\bsx_i}$ be the scrambled geometric net created using $\bsa_i$ and the permutations $\tilde{\pi}$. Similar to the definition in \eqref{eq:Wtilde} we define,
\begin{align*}
\tilde{W}^* =  \frac{1}{n\sigma_{sgn}} \sum_{i=1}^{n} \sum_{\tilde{\bsk} \succeq \tilde{m}\one , |\tilde{\bsk}| \ge m+1} \nu_{\tilde{\bsk}}(\tilde{\bsx}_i),
\end{align*}
where $\one$ denotes a vector of all ones of dimension $s$.
It is easy to see by our choice of $(I,J,\pi^{*})$, $(\tilde{W}, \tilde{W}^*)$ is an exchangeable pair of random variables. Now denote,
\begin{align*}
\tilde{S} &=  \frac{1}{n\sigma_{sgn}} \sum_{i=1}^{n} \sum_{\tilde{\bsk} \succeq \tilde{m}\one , |\tilde{\bsk}| \ge m+1}  \one(\Omega_{i,I,J,\tilde{m}})\nu_{\tilde{\bsk}}(\tilde{\bsx}_i),\\
S& =  \frac{1}{n\sigma_{sgn}} \sum_{i=1}^{n} \sum_{\tilde{\bsk} \succeq \tilde{m}\one , |\tilde{\bsk}| \ge m+1} \one(\Omega_{i,I,J,\tilde{m}})\nu_{\tilde{\bsk}}(\bsx_i),
\end{align*}
where $\Omega_{i,I,J,\tilde{m}}$ denotes the event that $(a_{I,J,1}, \ldots, a_{I,J,\tilde{m}-1}) = (a_{i,J,1}, \ldots, a_{i,J,\tilde{m}-1})$. Using this we write,
\begin{align*}
\tilde{W}^* - \tilde{W} &= \tilde{S} - S,\\
V &= \tilde{W} - S.
\end{align*}

Now, let $t \in \real$ and $h_t : \real \rightarrow \real$ be the unique bounded solution of $h^{'}(w) - w h(w) = \mathds{1}(w \le t) - \Phi(t)$. Using the fact that $(\tilde{W}, \tilde{W}^*)$ is exchangeable we have by arguments similar to Loh \cite{loh:2003},
\begin{align*}
0 &= \E\left((\tilde{W}^* - \tilde{W})[h_t(\tilde{W}) + h_t(\tilde{W}^*)]\right)\\
&= 2\E\left( h_t(\tilde{W}) \E(\tilde{W}^* - \tilde{W} \lvert \cw)\right) + \E\left((\tilde{W}^* - \tilde{W})[h_t(\tilde{W}^*) - h_t(\tilde{W})]\right),
\end{align*} 
where $\cw$ denote the $\sigma$-algebra generated by the random variables
\begin{align*}
\{ \pi_{j\sumdot}(a_{i,j,1}), \pi_{j\sumdot a_{i,j,1}}(a_{i,j,2}), \pi_{j\sumdot a_{i,j,1}, a_{i,j,2}}(a_{i,j,3}), \ldots \;: 1\le i \le n, \; 1 \le j \le s\}.
\end{align*}
Using arguments similar to Proposition 2 from \cite{loh:2003} we have
\begin{align}
\label{eq:expWgzW}
\E\left(\tilde{W} h_t(\tilde{W})\right) &= \frac{b^{\tilde{m} -1}}{2}\E\left((\tilde{W}^* - \tilde{W})[h_t(\tilde{W}^*) - h_t(\tilde{W})]\right) \nonumber \\
&= \E\left( \int h_t^{'}(V+ w) K_{\tilde{W},\tilde{W}^*}(w) \rd w\right),
\end{align}
where for all $w \in \real$,
\begin{align*}
K_{\tilde{W},\tilde{W}^*}(w) = \begin{cases}
\frac{b^{\tilde{m} -1}}{2}(\tilde{W}^* - \tilde{W}), & \text{ if } S < w \le \tilde{S},\\
\frac{b^{\tilde{m} -1}}{2}(\tilde{W} - \tilde{W}^*), & \text{ if } \tilde{S} < w \le S,\\
0 & \text{otherwise}.
\end{cases}
\end{align*}
Now from Lemma \ref{lem:stein} and \eqref{eq:expWgzW} we have,
\begin{align*}
\P(\tilde{W} \le t) - \Phi(t) &= \E\left(h_t^{'}(\tilde{W}) - \tilde{W}h_t(\tilde{W})\right)\\
&= \E \left( \int [h_t^{'}(\tilde{W}) - h_t^{'}(V+w)]K_{\tilde{W},\tilde{W}^*}(w) \rd w \right) \\
& \;\;\; + \E\left(h_t^{'}(\tilde{W})\right) E\left(\int K_{\tilde{W},\tilde{W}^*}(w) \rd w\right)\\
& \;\;\; - \E\left(h_t^{'}(\tilde{W}) \int K_{\tilde{W},\tilde{W}^*}(w) \rd w\right)\\
& \;\;\; + \E\left(h_t^{'}(\tilde{W})\right) \left( 1 - \E \int K_{\tilde{W},\tilde{W}^*}(w) \rd w\right).
\end{align*}
The proof follows by appropriately bounding the terms on the right hand side of the above equation. Propositions 3-5 of \cite{loh:2003} gives the appropriate bound for each of the terms when underlying point set is a $(0,m,s)$-scrambled digital net. Following the proofs in \cite{loh:2003} it is can be seen that the properties of a scrambled net are used to only prove Lemma 3, which is a supporting lemma to Proposition 5. 

Hence, to carry over the argument to $(0,m,s)$-scrambled geometric nets we prove a corresponding version of Lemma 3 of \cite{loh:2003} as Lemma \ref{lem:cor_lemma_3} in the Appendix \ref{app_thm}.
Finally, using Lemma \ref{lem:cor_lemma_3} and Propositions 3-5 in \cite{loh:2003} we have as $m \rightarrow \infty$,
\begin{align*}
\sup \left\{ |P(\tilde{W} \le w) - \Phi(w) | : -\infty < w < \infty \right\} &= O\left(\left(\frac{\tilde{m}}{m}\right)^{1/2}\right) + O\left(b^{-\tilde{m}/3}\right) \\
&= O \left( \left(\frac{\log_b(m)}{m}\right)^{1/2}\right).
\end{align*}
This completes the proof of Theorem \ref{thm:convergence}.

\begin{remark} Note that the above result does not provide a rate of convergence to normality for $W$. As in \cite{loh:2003}, it only gives a rate of convergence for $\tilde{W}$.
\end{remark}

\section{Discussion}
Our results on lower bound on variance and thereafter proof of asymptotic normality of scrambled geometric net quadrature are obtained modulo certain smoothness assumptions on the function and properties of the domain $\cx$. The properties of the domain enter crucially in suitably bounding the smallest eigenvalue of a nonnegative definite matrix away from zero. It is an interesting open question to understand whether one can characterize spaces for which such a condition holds. In particular, the case when $\cx$ is a disk, then the adaptive partitioning scheme in base $b=2$, as suggested by Basu and Owen \cite{basu2015scrambled}, fails to satisfy the lower bound on the eigenvalue. However, as we noted in the example provided in Remark 3 following Theorem 1, the lower bound on the eigenvalue is not necessary for the lower bound on $\sigma_{sgn}^2$ to hold. Therefore, in order to prove a desired lower bound on the variance, one needs to use the explicit spectral decomposition of matrix $A_j$ arising in Theorem 1 for the case of a disk, before performing subsequent analysis. Although our simulation results show such a lower bound on $\sigma^2_{sgn}$ to be true, the exact theoretical analysis is cumbersome.

On the other hand, the parallel adaptive partitioning scheme for $b=4$, does not satisfy the sphericity condition used to prove the upper bound in \cite{basu2015scrambled}. Therefore, we do not proceed to prove a lower bound in this case. This interesting dependence of the problem on the base $b$ used for the construction of the scrambled geometric net, makes us believe that the study of the disk deserves separate special attention. Another example considered in Basu and Owen [2015], is that of a spherical triangle. For the spherical triangle, we believe at an intuitive level, that the condition on the eigenvalue, as required by Theorem 1, holds while using base $b=4$. Indeed, it is not too difficult to show that it is enough to have the condition on eigenvalue in Theorem  to hold for sufficiently large $k_j$ for $j=1,\ldots, s$. However, when $k_j$ is large enough, a split in the spherical triangle ``resembles" a triangle in the plane, and intuitively the required bound on the eigenvalues of the corresponding matrix can be obtained by suitable perturbation bounds for matrix eigenvalues. Apart from conditions on the domains considered, our smoothness assumptions on the function $f$ are not necessarily sharp. Bridging these gaps are goals of our future research.

\section*{Acknowledgement}
We would like to sincerely thank Prof. Art Owen for his support and discussions. We would also like to thank the anonymous referees whose comments have improved the paper substantially.

\bibliographystyle{imsart-nameyear}
\bibliography{qmc}

\appendix

\section{Lower Bound on Gain Coefficients}\label{app_lb_gain}
There has been some lower bound results in an unpublished report by Owen \cite{owen2002necessity} which we formalize here in Lemmas \ref{odd_growth} and \ref{even_growth}. From the results in \cite{owensinum} it is known that $\Gamma_{u,\kappa} = 0$ when $|u| + |\kappa|  \le m - t$ and $\Gamma_{u,\kappa} = 1$ for $|\kappa| \ge m$. Since we are working with the $(0, m, s)$-net it is known that
\begin{align}
\label{gamma_notation}
\Gamma_{u,k} =  1 + (1- b)^{-u} \left [ (-b)^{m-k}{u-1 \choose m-k}  - \sum_{j=0}^{m-k}{u \choose j}(-b)^j\right].
\end{align}
From here on end, we replace $|u|$ and $|\kappa|$ by $u$ and $k$ respectively in the expression for $\Gamma$. This simplifies the expression and there is no loss of information since, $\Gamma_{u,\kappa}$ depends only on the cardinality of $u$ and the sum of the components of $\kappa$.

\subsection{Supporting Lemmas}\label{app_lemmas}

The proof of the Lemma \ref{lem:lower_bound_gamma} depends on the following two key important observations. If $a$ and $0 \le r \le u$ are integers with $a+r$ even then,
\begin{align}
\label{even_obs}
0 \le {u \choose r} b^r - {u \choose r-1}b^{r-1} \le (-1)^a \sum_{j=0}^r {u \choose j} (-b)^j \le {u \choose r}b^j.
\end{align}
Similarly, if $a+r$ is odd with $1 \le r \le u$
\begin{align}
\label{odd_obs}
0 \ge -{u \choose r} b^r + {u \choose r-1}b^{r-1} \ge (-1)^a \sum_{j=0}^r {u \choose j} (-b)^j \ge -{u \choose r}b^j.
\end{align}
To prove a lower bound on $\Gamma_{u,k}$ we begin with a few lemmas.

\begin{lemma}
\label{odd_growth}
Let $L$ be any positive integer and $u = m - k + 2L -1$. Then, under the above notation, 
$$\Gamma_{m-k+2L - 1,k} \ge 1.$$ 
\end{lemma}

\begin{proof}
Let $u = m - k + 2L - 1$ for $L \ge 1$. From \eqref{gamma_notation}
\begin{align*}
\Gamma_{u,k} &=  1 + (1- b)^{-u} \left [ (-b)^{m-k}{u-1 \choose m-k}  - \sum_{j=0}^{m-k}{u \choose j}(-b)^j\right]\\
&=  1 + (b - 1)^{-u} \left [ (-1)^{m-k+u}b^{m-k}{u-1 \choose m-k}  + (-1)^{u+1} \sum_{j=0}^{m-k}{u \choose j}(-b)^j\right]\\
&\ge 1 + (b - 1)^{-u} \left [ -b^{m-k}{u-1 \choose m-k}  + {u \choose m- k} b^{m-k} - {u \choose m-k-1}b^{m-k-1}\right],
\end{align*}
where the last line follows from \eqref{even_obs} since $u+1+m-k$ is even. Simplifying the above using Pascal's identity, we get
\begin{align*}
\Gamma_{u,k} &\ge  1 + (b-1)^{-u}b^{m-k-1} {u - 1 \choose m-k-1} \left(b - \frac{u}{2L}\right) \\
&\ge 1 + (b-1)^{-u}b^{m-k-1} {u - 1 \choose m-k-1} \left(b - \frac{b}{2}\right) \\
&\ge 1,
\end{align*}
 where the second inequality follows from $b\ge s \ge u$ and $L \ge 1$.
\end{proof}

\begin{lemma}
\label{even_growth}
Let $L$ be any positive integer and $u = m - k + 2L$. Then, under the above notation, 
$$\Gamma_{u,k} \le \Gamma_{u+2,k}.$$
\end{lemma}

\begin{proof}
Let $u = m - k + 2L$ for $ L \ge 1$. Then using the fact that $u + m -k$ is even and equations \eqref{even_obs} and \eqref{odd_obs} we have,
\begin{align*}
\Gamma_{u,k} - \Gamma_{u+2,k} &= (-1)^u(b - 1)^{-u} \left [ (-1)^{m-k}b^{m-k}{u-1 \choose m-k}  - \sum_{j=0}^{m-k}{u \choose j}(-b)^j\right] \\
&\;\;\; - (-1)^{u+2}(b - 1)^{-u-2} \left [ (-1)^{m-k}b^{m-k}{u+1 \choose m-k}  - \sum_{j=0}^{m-k}{u + 2 \choose j}(-b)^j\right] \\
&= (b-1)^{u-2}\Bigg[ (b-1)^2b^{m-k}{u-1 \choose m-k} - b^{m-k}{u+1 \choose m-k}  \\
&\qquad\qquad\qquad + (b-1)^2(-1)^{u+1} \sum_{j=0}^{m-k}{u \choose j}(-b)^j + (-1)^u\sum_{j=0}^{m-k}{u + 2 \choose j}(-b)^j\Bigg]\\
&\le (b-1)^{u-2}\Bigg[ (b-1)^2b^{m-k}{u-1 \choose m-k} - b^{m-k}{u+1 \choose m-k}  \\
&\qquad\qquad\qquad + (b-1)^2\left[{u \choose m-k-1}b^{m-k-1} - {u \choose m-k)}b^{m-k}\right]\\
&\qquad\qquad\qquad  + {u+2 \choose m-k}b^{m-k}\Bigg].
\end{align*}
Multiplying both sides by $(b-1)^{u+2}b^{-(m-k-1)}$ and applying the Pascal's identity we have,
\begin{align*}
\frac{(b-1)^{u+2}}{b^{(m-k-1)}}\left( \Gamma_{u,k} - \Gamma_{u+2,k} \right) &\le b(b-1)^2{u-1 \choose m-k} - b{u+1\choose m-k} + (b-1)^2{u \choose m-k-1}\\
& \;\;\; - b(b-1)^2{u \choose m-k} + b{u+2 \choose m-k} \\
&= (b-1)^2{u \choose m-k-1} - b(b-1)^2 {u-1 \choose m-k-1} + b{u+1 \choose m-k-1}\\
&= {u-1 \choose m-k-1} \left[ (b-1)^2 \frac{u}{2L + 1} + \frac{bu(u+1)}{(2L+1)(2L+2)} - b(b-1)^2\right].
\end{align*}
Now note that 
\begin{align}
\label{bd1}
\frac{u}{2L+1} \le \frac{b}{2L+1} \le \frac{b}{3} \;\;\text{and}\;\; \frac{u(u+1)}{(2L+1)(2L+2)} \le \frac{b(b+1)}{12}.
\end{align}
Using \eqref{bd1} and writing 
$$C := \frac{(b-1)^{u+2}b^{-(m-k-1)}}{{u-1 \choose m-k-1}},$$
we have,
\begin{align*}
C\left( \Gamma_{u,k} - \Gamma_{u+2,k} \right) &\le \frac{b(b-1)^2}{3} + \frac{b^2(b+1)}{12} - b(b-1)^2\\
&= \frac{b}{12} \left( b(b+1) - 8(b-1)^2 \right).
\end{align*}
Now the function $f(b) = b(b+1) - 8(b-1)^2$ has roots $b = 0.638$ and $b = 1.79$. Thus for all $b \ge 2$, we have $f(b) \le 0$, which finishes the proof of the lemma. 
\end{proof}

\subsection{Proof of Lemma \ref{lem:lower_bound_gamma}}
\begin{proof}[Proof of Lemma \ref{lem:lower_bound_gamma}]
Following Lemma \ref{odd_growth} and \ref{even_growth}, it is easy to see that the minimum value of $\Gamma_{u,k}$ is obtained at $u = m-k+2$. Plugging this in \eqref{gamma_notation} we get,
\begin{align*}
\Gamma_{m-k+2,k} = \left(\frac{b}{b-1}\right)^{m-k} \left(1 - \frac{m-k}{b-1}\right).
\end{align*}
Thus, 
\begin{align}
\label{bd2}
\Gamma_{u,k} \ge \min\limits_{\substack{0 \le k \le m\\ m-k+2 \le s}}\left(\frac{b}{b-1}\right)^{m-k} \left(1 - \frac{m-k}{b-1}\right).
\end{align}
To prove the result, we consider the following function for $x \in [0,m]$,  
\begin{align*}
f(x) = \left(\frac{b}{b-1}\right)^{m-x} \left(1 - \frac{m-x}{b-1}\right),
\end{align*}
and prove that is an increasing function of $x$. This will be enough to show that at $k = \max\{0, m-s+2\}$ the right hand side of \eqref{bd2} attains it minimum and therefore the theorem will follow. Towards that end, we differentiate $f$ with respect to $x$ to get
\begin{align}
f'(x) = \left(\frac{b}{b-1}\right)^{m-x} \left[ \frac{1}{b-1} - \left(1 - \frac{m-x}{b-1}\right)\log\left(\frac{b}{b-1}\right)\right].
\end{align}
Observing that $x \le m$ and $\log x \le x - 1$ for $x > 0$, we have $f'(x) > 0$ for all $x \in [0,m]$. Thus, $f(x)$ is an increasing function of $x$ and the proof is complete.
\end{proof}

\section{Proof of Technical Lemmas}
\label{app:lemmas}

\subsection{Supporting lemma for Theorem \ref{thm:convergence}}
\label{app_thm}
\begin{lemma}
\label{lem:cor_lemma_3}
Under the notation of Section \ref{sec:mrotos} we have,
\begin{align}
\label{eq:up_bound_fpsi}
\left|\langle f, \psi_{u\kappa\tau\gamma} \rangle\right| \le C b^{-\frac{|\kappa|}{2}\left(1+ \frac{2}{d}\right) - \frac{|u|}{2}} \sum_{\eta \in S_u} \sup_{\bsw \in \cx_{(u,\kappa,\tau)}} |\partial^\eta f_u(\bsw)|,
\end{align}
where $S_u = \{1, \ldots, d\}^{|u|}$ and $\cx_{(u,\kappa,\tau)} = \prod_{j \in u} \cx_{(k_j, t_j)}^{(j)}$. Furthermore, we have,
\begin{align}
\label{eq:up_bound_siguk}
\sigma_{u,\kappa}^2 = \E(\nu_{u,\kappa}^2(X)) \le C b^{-2|\kappa|/d} \norm{h_u}_\infty^2,
\end{align}
where $h_u(\bsz) = \max_{\eta \in S_u} |\partial^\eta \tilde{f}_u(\bsz)|$ and $\tilde{f}_u$ is the extension of $f_u$ to the bounding box. Finally we have,
\begin{align}
\label{eq:fourth_moment}
\E(\nu_{u,\kappa}^4) = O(b^{-4|\kappa|/d}).
\end{align}
\end{lemma}

\begin{proof}
The proof of \eqref{eq:up_bound_fpsi} easily follows from Lemma \ref{lem:inn_prod_s} and the sphericity constraint on the domain $\cx$. To prove \eqref{eq:up_bound_siguk} we refer to Lemma 5 of \cite{basu2015scrambled}. Now note that
\begin{align}
\label{eq:fourth_mom1}
\E(\nu_{u,\kappa}^4(X)) = \sum_{\tau} \sum_{\gamma_1, \gamma_2, \gamma_3, \gamma_4} \left(\prod_{i=1}^4 \langle f, \psi_{u\kappa\tau\gamma_i} \rangle \right) \int_{\cx^{u}} \prod_{i=1}^4
\psi_{u\kappa\tau\gamma_i}(\bsx) \rd \bsx,
\end{align}
and
\begin{align}
\label{eq:fourth_mom2}
 \int_{\cx^{u}} \prod_{i=1}^4
\psi_{u\kappa\tau\gamma_i}(\bsx) \rd \bsx &= \prod_{j \in u} \int  \prod_{i=1}^4 b^{(k_j - 1)/2} \left(b N_{j,k_j,t_j, c_{ji}}(x_j) - W_{k_j, t_j}(x_j)\right) \rd x_j \nonumber \\
&= \prod_{j \in u} O(b^{k_j}) = O(b^{|\kappa|}).
\end{align}
Thus, the proof of \eqref{eq:fourth_moment} now follows from \eqref{eq:up_bound_fpsi}, \eqref{eq:fourth_mom1} and \eqref{eq:fourth_mom2}. 
\end{proof}

\subsection{Proof of supporting lemmas for Theorem \ref{thm:lower-bound}}
\label{app_lb_sig}
\begin{proof}[Proof of Lemma \ref{rule1}]
By the multivariate Taylor expansion we have,
\begin{align*}
f(\bsx) &= f(\bsx^*) + \int_0^1 \langle \nabla f(\bsx^* + t(\bsx - \bsx^*)), \bsx - \bsx^* \rangle \rd t\\
&= f(\bsx^*) + \langle \nabla f(\bsx^*), \bsx - \bsx^*\rangle + \int_0^1 \langle \nabla f(\bsx^* + t(\bsx - \bsx^*)) - \nabla f(\bsx^*), \bsx - \bsx^* \rangle \rd t.
\end{align*}
Now, using Cauchy-Schwarz inequality we get,
\begin{align*}
\Bigg|\int_0^1 \langle \nabla f(\bsx^* + &t(\bsx - \bsx^*)) - \nabla f(\bsx^*), \bsx - \bsx^* \rangle \rd t\Bigg| \\
&\le \int_0^1 \left|\langle \nabla f(\bsx^* + t(\bsx - \bsx^*)) - \nabla f(\bsx^*), \bsx - \bsx^* \rangle\right| \rd t \\
&\le  \int_0^1 \norm{\nabla f(\bsx^* + t(\bsx - \bsx^*)) - \nabla f(\bsx^*)} \norm{ \bsx - \bsx^* } \rd t\\
&\le \int_0^1 B \norm{t(\bsx - \bsx^*)}^\beta \norm{ \bsx - \bsx^* } \rd t\\
&= B(1 + \beta)^{-1}\norm{\bsx - \bsx^*}^{1+ \beta}.
\end{align*}
\end{proof}

\begin{proof}[Proof of Lemma \ref{rule2}]
Taking $\bsx^* = \bsn_{ktc}$ in Lemma \ref{rule1} we get,
\begin{align*}
b^{k+1}\int_{\cx_{(k,t,c)}} f(\bsx) \rd \bsx &= f(\bsn_{ktc}) + b^{k+1}\int_{\cx_{(k,t,c)}} \langle \nabla f(\bsn_{ktc}), \bsx - \bsn_{ktc}\rangle \rd \bsx \\
&\;\;\;+ b^{k+1}\int_{\cx_{(k,t,c)}}C \norm{ \bsx - \bsn_{ktc}}^{1+\beta} \rd \bsx \\ 
& = f(\bsn_{ktc}) + O(b^{-k(1+\beta)/d}),
\end{align*}
where the last equality follows from the fact that $$b^{k+1}\int_{\cx_{(k,t,c)}} \langle \nabla f(\bsn_{ktc}), \bsx - \bsn_{ktc}\rangle \rd \bsx =0.$$
Applying the same with $\bsw_{kt}$, taking the difference and using Lemma \ref{rule1} we have,
\begin{align*}
b^{k+1}\int_{\cx_{(k,t,c)}} f(\bsx) \rd \bsx - b^{k}\int_{\cx_{(k,t)}} f(\bsx) \rd \bsx 
&= f(\bsn_{ktc}) - f(\bsw_{kt}) + O(b^{-k(1+\beta)/d})\\
&=  \langle \bsn_{ktc} - \bsw_{kt}, \nabla f(\bsw_{kt})\rangle + O(b^{-k(1+\beta)/d}).
\end{align*}
\end{proof}

\begin{proof}[Proof of Lemma \ref{lem:fu_smooth}]
We prove this by induction on $|u|$. Let $|u| = 0$, that is $u = \emptyset$. 
Then $f_u(\bsx) = \int_{\cx^\otos} f(\bsx) \rd \bsx$ which is a constant $\mu$ 
and is therefore smooth on $\cx^\otos$. 
Let us suppose that the hypothesis holds for $|u| = k - 1<s$ and we shall prove it for $|u| = k$. 

Fix any $u \subseteq\otos$ such that $|u| = k$. By \eqref{eq:anovau} we have,
 \[f_u(\bsx) = \int_{\cx^{-u}} f(\bsx) \rd \bsx_{-u} - \sum_{w \subset u} f_w(\bsx).\]
Note that $f_u$ only depends only on $\bsx_u$. Thus for any $\bsx_u, \bsx_u^* \in \cx^{u}$ and any $v \subseteq u$, using Leibniz's integral rule and the induction hypothesis we have,
\begin{align*}
\norm{\nabla^v f_u(\bsx_u) - \nabla^v f_u(\bsx_u^*)} &\le \norm{ \nabla^v  \int_{\cx^{-u}} f(\bsx_u, \bsx_{-u}) \rd \bsx_{-u} - \nabla^v  \int_{\cx^{-u}} f(\bsx_u^*, \bsx_{-u}) \rd \bsx_{-u} } \\
&\;\;\; + \sum_{w \subset u} \norm{\nabla^v f_w(\bsx_u) - \nabla^v f_w(\bsx_u^*)} \\
& \le  \int_{\cx^{-u}} \norm{\nabla^v f(\bsx_u, \bsx_{-u}) - \nabla^v f(\bsx_u^*,\bsx_{-u})} \rd \bsx_{-u} \\
&\;\;\; + \sum_{w \subset u} \norm{\nabla^v f_w(\bsx_u) - \nabla^v f_w(\bsx_u^*)} \\
&\le \tilde{B}\norm{\bsx_u - \bsx_u^*}^\beta.
\end{align*}
Thus the induction hypothesis holds for $|u|=k$, and hence completing the proof.
\end{proof}

\begin{proof}[Proof of Lemma \ref{lem:inn_prod_s}]
Note that from the definition we have,
\begin{align*}
\langle f, \psi_{u\kappa\tau\gamma} \rangle &= \langle f_u, \psi_{u\kappa\tau\gamma} \rangle\\
&= b^{-(|\kappa| + |u|)/2} \int_{\cx^u} f_u(\bsx) \prod_{j \in u} b^{k_j}\left(b N_{jk_jt_jc_j}(x_j) - W_{jk_jt_j}(x_j)\right) \rd \bsx_u,
\end{align*}
where $N_{jk_jt_jc_j}, W_{jk_jt_j}$ are indicator functions of $\cx_{(k_j,t_j,c_j)}^{(j)}, \cx_{(k_j,t_j)}^{(j)}$ respectively. 
Since $f_u$ is smooth by Lemma \ref{lem:fu_smooth}, we apply Lemma \ref{rule2}, $|u|$ times to get,
 \begin{align*}
\langle f, \psi_{u\kappa\tau\gamma} \rangle &=  b^{-(|\kappa| + |u|)/2} \sum_{i_1=1}^d \sum_{i_2 = 1}^d \ldots \sum_{i_{|u|} = 1}^d \left(\prod_{j \in u}(\bsn_{c_j} - \bsw_j)_{i_j}\right) \frac{\partial^{|u|}f_u (\bsw_1, \ldots, \bsw_{|u|})}{\prod_{j \in u} \partial x_{j,i_j}}  \\
&\;\;\; +  b^{-(|\kappa| + |u|)/2}  O(b^{-|\kappa|/d - \tilde{k}\beta/d}) \\
&= b^{-(|\kappa| + |u|)/2} \sum_{i_1, i_2, \ldots, i_{|u|}=1}^d  \left(\prod_{j \in u}(\bsn_{c_j} - \bsw_j)_{i_j}\right) \frac{\partial^{|u|}f_u (\bsw_1, \ldots, \bsw_{|u|})}{\prod_{j \in u} \partial x_{j,i_j}} \\
&\;\;\; +  O\left(b^{-\frac{|\kappa|}{2}\left(1+\frac{2}{d}\right) - \frac{\tilde{k}\beta}{d}}\right)\\
& = b^{-(|\kappa| + |u|)/2} \sum_{ \eta \in S_u} \left(\prod_{j \in u}(\bsn_{c_j} - \bsw_j)_{\eta_j}\right) \partial^{\eta} f_u( \bsw) + O\left(b^{-\frac{|\kappa|}{2}\left(1+\frac{2}{d}\right) - \frac{\tilde{k}\beta}{d}}\right).
\end{align*}
Note that the order is retained only for the last integral since for every other order term, integration with respect to $\psi_{ktc}$ is zero. 
\end{proof}

\begin{proof}[Proof of Lemma \ref{sigmauk}]
Let $u \subseteq \{1, \ldots, s \}$. From the definition we have,
\begin{align*}
\sigma_{u,\kappa}^2 &= \int_{\cx^{u}} \nu_k^2(\bsx) \rd \bsx \\
&= \sum_{\tau}\sum_{\gamma,\gamma'} \langle f, \psi_{u\kappa\tau\gamma} \rangle \langle f, \psi_{u\kappa\tau\gamma'} \rangle \int_{\cx^u} \psi_{u\kappa\tau\gamma}(\bsx)  \psi_{u\kappa\tau\gamma'} (\bsx) \rd \bsx_u \\
&= \sum_{\tau}\sum_{\gamma,\gamma'} \langle f, \psi_{u\kappa\tau\gamma} \rangle \langle f, \psi_{u\kappa\tau\gamma'} \rangle \prod_{j \in u} \left(\mathds{1}_{c_j=c_j'} - \frac{1}{b}\right).
\end{align*}
Note that $S_u$ depends only on the cardinality of $u$. Thus we can write, $\sum_{\eta \in S_u}$ as $\sum_{i=1}^d \sum_{\eta \in S_{u \setminus u_1}}$ by separating out the first component of $u$. 
Now using Lemma \ref{lem:inn_prod_s} we have,
\begin{align*}
&b^{(|\kappa| + |u|)} \langle f, \psi_{u\kappa\tau\gamma} \rangle \langle f, \psi_{u\kappa\tau\gamma'} \rangle \\ 
&= \sum_{ \eta \in S_u} \sum_{ \eta' \in S_u} \left(\prod_{j \in u}(\bsn_{c_j} - \bsw_j)_{\eta_j} (\bsn_{c_j'} - \bsw_j)_{\eta_j'}\right) \partial^{\eta} f_u( \bsw) \partial^{\eta'} f_u( \bsw) + O\left(b^{-2|\kappa|/d - \tilde{k}\beta/d}\right)  \\
&= \sum_{i=1}^d \sum_{j=1}^d (\bsn_{c_{u_1}} - \bsw_{u_1})_{i} (\bsn_{c_{u_1}'} - \bsw_{u_1})_{j}\sum_{\eta \in S_{u \setminus u_1}} \sum_{\eta' \in S_{u \setminus u_1}} \left(\prod_{j \in u \setminus u_1}(\bsn_{c_j} - \bsw_j)_{\eta_j} (\bsn_{c_j'} - \bsw_j)_{\eta_j'}\right) \\
& \;\;\; \times \partial^{i : \eta} f_u( \bsw) \partial^{j : \eta'} f_u( \bsw)+ O\left(b^{-2|\kappa|/d - \tilde{k}\beta/d}\right). 
\end{align*}
Plugging this into equation for $\sigma_{u,\kappa}^2$ we have,
\begin{align*}
&b^{(|\kappa| + |u|)}\sigma_{u,\kappa}^2 \\
&= \sum_{\tau}\sum_{\gamma,\gamma'} \sum_{i=1}^d \sum_{j=1}^d (\bsn_{c_{u_1}} - \bsw_{u_1})_{i} (\bsn_{c_{u_1}'} - \bsw_{u_1})_{j} \left(\mathds{1}_{c_{u_1}=c_{u_1}'} - \frac{1}{b}\right)\\
&\;\;\; \times \sum_{\eta \in S_{u \setminus u_1}} \sum_{\eta' \in S_{u \setminus u_1}} \left(\prod_{j \in u \setminus u_1}(\bsn_{c_j} - \bsw_j)_{\eta_j} (\bsn_{c_j'} - \bsw_j)_{\eta_j'} \left(\mathds{1}_{c_j=c_j'} - \frac{1}{b}\right) \right) \\
&\;\;\; \times\partial^{i : \eta} f_u( \bsw) \partial^{j : \eta'} f_u( \bsw) +O\left(b^{|\kappa| -2|\kappa|/d - \tilde{k}\beta/d}\right) .
\end{align*}
Note that sum over $\gamma, \gamma'$ is the sum over $c_{u_1}, \ldots, c_{u_{|u|}}$ and $c_{u_1}', \ldots, c_{u_{|u|}}'$ respectively.  For notational simplicity let us define $\gamma_{-1}$ and $\gamma_{-1}'$ as the collection $c_{u_2}, \ldots, c_{u_{|u|}}$ and $c_{u_2}', \ldots, c_{u_{|u|}}'$ respectively, i.e. leaving out $c_{u_1}$ and $c_{u_1}'$. Further define $A_{u_1}$ as the $d\times d$ matrix whose $i,j$-th entry is given by
$$a_{ij} := \sum_{c_{u_1}, c_{u_1}'} (\bsn_{c_{u_1}} - \bsw_{u_1})_{i} (\bsn_{c_{u_1}'} - \bsw_{u_1})_{j} \left(\mathds{1}_{c_{u_1}=c_{u_1}'} - \frac{1}{b}\right),$$
and $A_{u \setminus u_1}$ be the $d^{|u| - 1} \times d^{|u| -1}$ matrix whose rows and columns are chronologically numbered by $\eta$ and $\eta' \in S_{u \setminus u_1}$ respectively and a typical entry is of the matrix is given by
\begin{align}
\label{matA}
\sum_{\gamma_{-1}, \gamma_{-1}' } \prod_{j \in u \setminus u_1}(\bsn_{c_j} - \bsw_j)_{\eta_j} (\bsn_{c_j'} - \bsw_j)_{\eta_j'} \left(\mathds{1}_{c_j=c_j'} - \frac{1}{b}\right).
\end{align}
Finally we write $\partial^\eta f$ for $\eta \in S_{u \setminus u_1}$ as the vector $\nabla^{u \setminus u_1} f$ having dimension $d^{|u|-1}$. Combining these, we have
\begin{align*}
b^{(|\kappa| + |u|)}\sigma_{u,\kappa}^2 &= \sum_{\tau} \sum_{i,j=1}^d a_{ij} ( \partial^i \nabla^{u \setminus u_1} f_u(\bsw) )^{T} A_{u \setminus u_1} ( \partial^j \nabla^{u \setminus u_1} f_u (\bsw) ) + O\left(b^{|\kappa| -2|\kappa|/d - \tilde{k}\beta/d}\right)  \\
&= \sum_{\tau} (\nabla^u f_u(\bsw))^T ( A_{u_1}  \otimes A_{u \setminus u_1} ) \nabla^{u} f_u (\bsw) + O\left(b^{|\kappa| -2|\kappa|/d - \tilde{k}\beta/d}\right). 
\end{align*}
Applying this recursively, we get,
\begin{align*}
b^{(|\kappa| + |u|)}\sigma_{u,\kappa}^2 &= \sum_{\tau} (\nabla^u f_u(\bsw))^T \left(\bigotimes_{j \in u} A_j \right) \nabla^{u} f_u (\bsw) + O\left(b^{|\kappa| -2|\kappa|/d - \tilde{k}\beta/d}\right). 
\end{align*}
\end{proof}

\section{Alternative argument on Triangles}
\label{app:lower_bound_triangles}
\subsection*{Proof of Lemma \ref{lem:sig_form}}

	From the proof of Lemma \ref{sigmauk}, it is easy to see that
	\begin{align*}
	\sigma_{k}^2 =\sum_{t = 0}^{b^k - 1}\sum_{c,c'}\langle f, \psi_{ktc}\rangle \langle f, \psi_{ktc'}\rangle (1_{c=c'} - b^{-1}). 
	\end{align*}
	Consider the matrix $A_{b\times b}= I_{b} - J_{b}/b$, where $I_b$ denotes the identity matrix of dimension $b$ and $J_b$ is the matrix of all ones. Let us define $\bsy_t$ as the column vector whose $i$-th entry is $\langle f, \psi_{kt{i-1}} \rangle$ for $i = 1, \ldots, b$. Under this transformation, it is easy to see that
	\begin{align}
	\label{newsigma_exp}
	\sigma_{k}^2 = \sum_{t = 0}^{b^k - 1} \bsy_t^{T}A\bsy_t.
	\end{align}
	Denote by $\bsv_b$ the vector of all ones, properly normalized. Then from the structure of $A$, we can see $Av_b = 0$. Furthermore let $\bsv_1, \ldots, \bsv_{b-1}$ be the set of orthonormal eigen vectors of $A$ which are orthogonal to the eigen space of 0. Note further that because of the structure of $A$, $A\bsv_\ell = \bsv_\ell$ for $\ell = 1,\ldots, b-1$.

	Thus $\bsv_1, \ldots, \bsv_b$ forms a orthonormal basis of $\mathbb{R}^b$ and we expand $\bsy_t$ along this basis to get,
	\begin{align}
	\label{bsy}
	\bsy_t = \sum_{\ell = 1}^{b-1}a_\ell^t\bsv_\ell + a_b^t\bsv_b.
	\end{align}   
	Plugging \eqref{bsy} in \eqref{newsigma_exp} we get,
	\begin{align}
	\label{sig_final}
	\sigma_{k}^2 &= \sum_{t = 0}^{b^k - 1} \bsy_t^{T}A\bsy_t = \sum_{t = 0}^{b^k - 1}\left(\sum_{\ell = 1}^{b-1}a_\ell^t\bsv_\ell \right)^T A \left(\sum_{\ell = 1}^{b-1}a_\ell^t\bsv_\ell\right) = \sum_{t = 0}^{b^k - 1}\sum_{\ell = 1}^{b-1}(a_\ell^t)^2
	\end{align}
	Let us pick one particular choice of eigen vectors. Define,
	$$\bsv_\ell := (1, \ldots, 1, -\ell, 0, \ldots, 0)/\sqrt{\ell(1 + \ell)} \;\; \text{ for } \;\;  \ell = 1, \ldots, b-1 ,$$ 
	where we have $\ell$ ones initially followed by $-\ell$ at position $\ell + 1$ and zeros for the remaining entries. It is easy to verify that $\bsv_1, \ldots, \bsv_{b-1}$ satisfies all the required properties. We derive a closed form expression for $a_\ell^t$ using these eigen vectors,
	\begin{align*}
	a_\ell^t &= \langle \bsy_t, \bsv_\ell \rangle \\
	&= \frac{1}{\sqrt{\ell(1 + \ell)}} \left(\sum_{i=1}^\ell \langle f, \psi_{kti-1}\rangle - \ell \langle f, \psi_{kt\ell}\rangle \right).
	\end{align*}
	Using $\langle f, \psi_{ktc}\rangle = b^{-(k+1)/2} \int f(\bsx) b^k(b N_{ktc} - \cx_{(k,t)}) d\bsx$, we get,
	\begin{align*}
	a_\ell^t &= \frac{ b^{-(k+1)/2} }{\sqrt{\ell(1 + \ell)}} \left(b^{k+1}\sum_{i=1}^\ell  \int_{\cx_{(k,t,i-1)}} f(\bsx)d\bsx  - \ell b^{k+1}\int_{\cx_{(k,t,\ell)}}  f(\bsx)d\bsx   \right)\\
	&= \frac{ b^{(k+1)/2} }{\sqrt{\ell(1 + \ell)}} \sum_{i=1}^\ell \left( \int_{\cx_{(k,t,i-1)}} f(\bsx)d\bsx  - \int_{\cx_{(k,t,\ell)}}  f(\bsx)d\bsx   \right).
	\end{align*}
	Plugging this into \eqref{sig_final} we have
	\[
	\sigma_{k}^2 =  \sum_{t = 1}^{b^k - 1} \sum_{\ell = 1}^{b-1} \frac{b^{k + 1}}{\ell(\ell+1)} \left[ \sum_{i = 1}^\ell \left( \int_{\cx_{(k,t,i-1)}} f(\bsx) d\bsx - \int_{\cx_{(k,t,\ell)}} f(\bsx) d\bsx \right)\right]^2.
	\]
$\hfill\square$

\end{document}

%% file: newcommandsnotbook.tex
\allowdisplaybreaks

\newcommand{\cf}{{\cal F}}
\newcommand{\cg}{{\cal G}}

\newcommand{\cs}{{\cal S}}

\newcommand{\cx}{{\cal X}}
\newcommand{\cw}{{\cal W}}

\newcommand{\ints}{\mathbb{Z}}

\newcommand{\bsa}{\boldsymbol{a}}
\newcommand{\bsb}{\boldsymbol{b}}

\newcommand{\bsu}{\boldsymbol{u}}
\newcommand{\bsx}{\boldsymbol{x}}
\newcommand{\bsy}{\boldsymbol{y}}
\newcommand{\bsz}{\boldsymbol{z}}

\newcommand{\natu}{\mathbb{N}}
\newcommand{\real}{\mathbb{R}}

\newcommand{\rd}{\mathrm{\, d}}

\newcommand{\one}{{\mathbf{1}}}
\newcommand{\bzero}{{\mathbf{0}}}

\newcommand{\var}{{\mathrm{Var}}}

\newcommand{\vol}{{\mathbf{vol}}}

\newcommand{\cc}{{\mathcal C}}

\renewcommand{\emptyset}{\varnothing}
\renewcommand{\ge}{\geqslant}
\renewcommand{\le}{\leqslant}

\newcommand{\dustd}{\mathbf{U}} 

  \newcommand{\BIT}{\begin{itemize}}
\newcommand{\EIT}{\end{itemize}}
\newcommand{\BNUM}{\begin{enumerate}}
\newcommand{\ENUM}{\end{enumerate}}

\renewcommand{\exp}[1]{\operatorname{exp}\left(#1\right)} 

\def\E{\mathbb{E}} 

\def\P{\mathbb{P}} 

\theoremstyle{definition}
\newtheorem{definition}{Definition}
\newtheorem{example}{Example}

\theoremstyle{plain}
\newtheorem{theorem}{Theorem}
\newtheorem{corollary}{Corollary}
\newtheorem{propo}{Proposition}
\newtheorem{lemma}{Lemma}
\theoremstyle{definition}